\documentclass[oneside]{amsart}
\usepackage{graphicx}
\usepackage{amsfonts}
\usepackage{floatrow}
\usepackage{url}
\usepackage{amssymb, amsfonts, amsthm, amsmath,mathdots}
\usepackage{xfrac}
\usepackage{enumerate}
\usepackage{verbatim}
\usepackage[all]{xy}
\usepackage[utf8]{inputenc}
\usepackage{bbm}
\usepackage{mathrsfs,mathtools}
\usepackage{listings}
\usepackage{breqn}
\usepackage{textcomp}
\usepackage{moreverb}
\usepackage{sverb}
\usepackage{algorithm,algpseudocode,algorithmicx}

\usepackage{todonotes}

\usepackage{natbib}

\renewenvironment{thebibliography}[1]{
  \begin{oldthebibliography}{#1}
    \setlength{\itemsep}{0.5em}
    \setlength{\parskip}{0em}
}
{
  \end{oldthebibliography}
}

\usepackage[]{xcolor}
\usepackage[colorlinks=true, allcolors=blue,linkcolor=blue,urlcolor=black,citecolor=blue,backref=page]{hyperref}

\usepackage{cleveref}

\usepackage{tikz,pgfplots}
\usepackage{tikz-cd}
\usetikzlibrary{decorations.markings,positioning}

\newtheorem{theorem}{Theorem}[section]

\newtheorem{lemma}[theorem]{Lemma}
\newtheorem{corollary}[theorem]{Corollary}
\newtheorem{proposition}[theorem]{Proposition}
\numberwithin{equation}{section}

\theoremstyle{definition}

\newtheorem{definition}[theorem]{Definition}

\newtheorem{examplex}[theorem]{Example}
\newenvironment{example}
{\pushQED{\qed}\examplex}
{\popQED\endexamplex}

\newtheorem{remarkx}[theorem]{Remark}
\newenvironment{remark}
{\pushQED{\qed}\remarkx}
{\popQED\endremarkx}

\newtheoremstyle{citing}
{}
{}
{\itshape}
{}
{\bfseries}
{\textbf{.}}
{.5em}
{\thmnote{#3}}
{\theoremstyle{citing}
}

\DeclareMathOperator{\GL}{GL}

\DeclareMathOperator{\im}{im}

\renewcommand{\d}{\mathrm{d}}

\newcommand{\calA}{\mathcal{A}}
\newcommand{\cA}{\mathcal{A}}

\newcommand{\cC}{\mathcal{C}}
\newcommand{\cM}{\mathcal{M}}

\newcommand{\cP}{\mathcal{P}}

\newcommand{\cI}{\mathcal{I}}

\newcommand{\QQ}{\mathbb{Q}}

\def\co{\colon\thinspace} 

\newcommand{\sprod}[1]{\langle #1 \rangle}

\usepackage{xcolor}

\usepackage{amsmath}
\usepackage{amsfonts}
\usepackage{amsthm}
\usepackage{amssymb}

\usepackage[utf8]{inputenc} 
\usepackage[T1]{fontenc}    
\usepackage{url}            
\usepackage{booktabs}       
\usepackage{amsfonts}       
\usepackage{nicefrac}       
\usepackage{microtype}      
\usepackage{xcolor}         

\usepackage{geometry}[top=2cm,bottom=2cm,left=2cm,right=2cm]

\usepackage{amsmath}
\usepackage{amsthm}
\usepackage{enumerate}
\usepackage{graphicx}
\usepackage{caption}
\usepackage{subcaption}
\usepackage{dsfont}
\usepackage{xfrac}
\usepackage{wrapfig}

\newcommand{\dna}{\Delta\kern-0.48em{\nabla}}

\newcommand{\calC}{\mathcal{C}}

\newcommand{\calI}{\mathcal{I}}

\newcommand{\calN}{\mathcal{N}}
\newcommand{\calV}{\mathcal{V}}

\newcommand{\calW}{\mathcal{W}}

\newcommand{\R}{\mathbb{R}}
\newcommand{\N}{\mathbb{N}}

\renewcommand{\d}{\mathbf{d}}

\newcommand{\sse}{\subseteq}

\newcommand{\abs}[1]{\vert #1 \vert}
\newcommand{\norm}[1]{\Vert #1 \Vert}
\newtheorem{thmA}{Theorem}

\newtheorem{thmB}{Theorem}

\newtheorem{corC}{Corollary}

\newcommand{\PL}{\mathrm{PL}}

\pgfplotsset{compat=1.18}

\title{
On the identifiability and semi-algebraicity of ReLU neuromanifolds}

\author{Axel Flinth}
\email{axel.flinth@umu.se}
\address{Department of Mathematics and Mathematical Statistics, Ume\r{a} University, 901 87 Ume\r{a}, Sweden }

\author{Stefano Mereta}
\email{stefano.mereta@cunef.edu}
\address{Departamento de Matem\'aticas, CUNEF Universidad, Calle de Almansa, 101, 28040, Madrid, Spain}

\author{Michele Pernice}
\email{mpernice@uw.edu}
\address{Department of Mathematics, University of Washington, Padelford, Seattle, WA 98195-4350, USA}

\begin{document}

\begin{abstract}
We study several aspects of the semi-algebraicity of the neuromanifold $\mathcal{M}_\d$ of a feedforward ReLU neural network. 
We prove that $\mathcal{M}_\d$ in general is not a semi-algebraic quotient of the space of parameters of the network, but it can be approximated (in a precise sense) by semi-algebraic spaces. We then consider the notion of \emph{identifiable} subsets of the space of parameters, where the network does not have any hidden symmetries. We prove that the maximal identifiable subset is semi-algebraic for any neural network with semi-algebraic activation functions. Lastly, we compute explicitly the identifiable subset of a shallow ReLU network and prove that for input dimensions $\ge 3$ and width $\ge 6$ it is not euclidean open.
\end{abstract}

\maketitle

\tableofcontents

\section*{Introduction}
A  machine learning model can be viewed as a map $\mathcal W \times \mathcal X \rightarrow \mathcal Y$
where $\mathcal W$ is the space of parameters, $\mathcal X$ the space of inputs and $\mathcal Y$ the space of outputs. Given a pair $(w,x) \in \mathcal W \times \mathcal X$ we will denote its image as $f_w(x) \in \mathcal Y$. Equivalently, we may view the model as a parametrization of functions $x\mapsto f_w(x)$.  The space of such functions $\mathcal M :=\{f_w \co \mathcal X \rightarrow \mathcal Y \mid w \in \mathcal W\}$
is called the \emph{neuromanifold} associated to the model.

The (supervised) training of a model is an optimization process over the neuromanifold, seeking a function in $\mathcal{M}$ that minimizes a chosen loss with respect to a target function.
In practical settings, though, this optimization does not happen on the neuromanifold, but on the space of parameters of the model, via the parametrization map. This can cause distortions: so called \emph{spurious} critical points, i.e.\ critical parameters for the composition of the parametrization map with the loss function which are not critical points of the loss function on the neuromanifold. Furthermore, singular and boundary points of the neuromanifold can be more prone to be critical points for the loss function, thus knowing the geometry of the latter would give insights on the training dynamic.

The theoretical algebro-geometric study of neuromanifolds has recently experienced a quick development, under the name of \emph{neuroalgebraic geometry} (see \cite{marchetti} for an overview). Thus, algebraic geometry fits in a broader toolkit from all areas of mathematics (probability, statistics, topology, convex geometry, functional analysis, category theory...) that can be exploited in the effort of unveiling the mysteries of the training process, improve its efficiency, reduce its impact on the existing energy infrastructure and the environment (see \cite{iea} for a detailed analysis of the matter), and in general have a better understanding of AI models and their behavior.

The neuroalgebraic geometry of neural networks with polynomial activation functions has seen substantial recent progress (see for example \cite{trager2019pure, kohn2022geometry,   shahverdi, kohn2024function, kohn, henry, 
 razor}). In this case the neuromanifold is an algebraic variety (or more generally a semi-algebraic space) and there is a precise dictionary between machine learning terminology and algebro-geometric invariants of the neuromanifold (see \cite[Table 1]{marchetti}).

For non-polynomial activation functions, little is known about the geometry of the neuromanifold. This motivates the present work, in which we study the semi-algebraicity of neuromanifolds associated to feedforward fully connected neural networks with ReLU activation functions, as well as their (hidden) symmetries.

\subsection*{ReLU neuromanifolds}
In the following, $L \in \N$ is the number of layers of a neural network and $d_k \in \N$ is the width of the $k$-th layer, for $k = 0, \dots , L$. With this notation, we say that the neural network considered is of architecture $\d := (d_0, \dots ,d_L)$. The dimension of its space of parameters is $M := \sum_{k=0}^L d_k(d_{k-1}+1)$, thus the neural network is a function 
\begin{equation}\label{eq:network}
    F_{\mathbf d} \co \R^M \times \R^{d_0} \rightarrow \R^{d_L}, \quad (w,x) \mapsto f_w(x).
\end{equation} 
Given $w \in \mathcal \R^M$ the function $f_w$ can be represented as a composition 
\[
    f_w = W_L \circ \sigma \circ W_{L-1} \circ \dots \circ \sigma \circ W_1,
\]
where $W_k$ is an affine transformation $\R^{d_{k-1}} \rightarrow \R^{d_k}$ and the activation function $\sigma$ is (unless specified otherwise) the rectified linear unit (ReLU for short) given by the coordinatewise application of $\max(x,0)$. 

Every function in the neuromanifold $\mathcal M_\d$ is a continuous piecewise linear (PL, for short) function $\R^{d_0} \rightarrow \R^{d_L}$ with finitely many linearity regions, thus $\mathcal M_\d$ is a subset of the space of continuous functions $\mathcal{C}^0(\R^{d_0}, \R^{d_L})$. Conversely, in \cite{zhang2018tropical} the authors prove that for any PL function $f \co \R^{n} \rightarrow \R^{m}$ there exist  $L \in \N$ and an architecture $\d =(d_0=n, d_1, \dots ,d_L=m)$ such that $f \in  \mathcal M_{\d}$.

Unless specified otherwise, we reduce to the case of zero biases (i.e.\ we will assume that the affine functions $W_k$ are in fact linear functions) and $d_L=1$. It is clear that results obtained for $d_L=1$ generalize to any width. Statements obtained for zero bias require a more careful effort in order to be generalized, but we restrict ourselves to this setting for simplicity of presentation.

In this setting, addressing the problem of identifiability (i.e.\ if parameters can be uniquely recovered up to trivial simmetries from their image via $\Phi$) for a ReLU neural network of aychitecture $\d$ is equivalent to the study of the fibers of the parametrization map
\[
    \Phi \co \R^M \to \PL \subset  C^{0}(\R^{d_0}) \quad \quad  w \mapsto f_w.
\]

\subsection*{Related works} The geometry of ReLU neuromanifolds has been already approached in \cite{vallin2023geometric} and \cite{masden2025algorithmic}. In \cite{hertrich2021towards} the authors study which functions lie in $\mathcal{M}_{\mathbf d}$, rather than studying global properties of this space. Several works have been devoted to the study of the number of linear regions of the functions produced by a ReLU (and more in general maxout) neural network (see for example \cite{montufarbengio, montufarren}). In \cite{AlexMontu} the authors study, among other things, the geometry of the (Zariski closure of) the image of the maps $\Phi_n$ obtained as composition of the parametrization map with the evaluation at a finite number of points. As convex piecewise linear functions are tropical signomial and dual to polyhedra, in \cite{montufarbrand} the authors attach several combinatorial objects to a ReLU neural network and study what the properties of these objects can tell about the network. In \cite{grigsby2025functional} and \cite{GriLinRol} the authors study, via topological and analytical methods, the existence and some properties of open subsets of the space of parameters where there are no \emph{hidden} symmetries.  These results are extended and improved in \cite{montufargrillo}.  The symmetries of ReLU shallow networks are characterized in \cite{ramakrishnan}, where a similar (and seemingly weaker) version of our Corollary \ref{cor:strongly-identifiable-Zariski} is proven. The symmetries of the three-layer case with biases get a detailed treatment in \cite{gegenfurtner2026symmetries}. The authors of the aforementioned works use Lie theory, polyhedral geometry and combinatorics, while we approach these problems with algebro-geometric techniques. Furthermore, we do not  focus only on the explicit realization of hidden symmetries for specific architectures, but we prove geometric properties of the subset of identifiable parameters for large classes of networks. 

\subsection*{Structure of the paper and results}
In Section \ref{sec:semi-algebraicity} we investigate the semi-algebraicity of the neuromanifold $\mathcal{M}_{\mathbf{d}}$, in the spirit of the program outlined in  \cite{marchetti}. We investigate it from two viewpoints: firstly we introduce the notion of \emph{pointwise semi-algebraicity} for maps to the space of continuous functions $C^0(\R^{d_0})$ and use this notion to argue that $\mathcal{M}_{\mathbf{d}}$ can be seen as the limit of semi-algebraic spaces; later, we focus on the abstract semi-algebraicity of the neuromanifold as a quotient of the space of parameters.

Even though a notion of semi-algebraic subsets has never been given for an infinite dimensional space such as $C^0(\R^{d_0})$, we hint that the neuromanifold should be seen as a \emph{pro-semi-algebraic} space: a categorical limit of semi-algebraic spaces similar to what the authors of \cite{AlexMontu} call \emph{output varieties}. Each one of these is obtained by fixing a finite input set and evaluating the functions in the neuromanifold at these points, it lies inside a finite dimensional ambient space and is semi-algebraic in the usual sense. We believe this observation has practical relevance in the experimental study of ReLU neuromanifolds, as we explain later. 

We pass on to addressing the abstract semi-algebraicity of the neuromanifold. Let us denote as $E_\Phi$ the equivalence relation induced on the space of parameters $\R^M$ by the parametrization map $\Phi$. It is a straightforward fact that the neuromanifold, as a set, is a quotient of the space of parameters by $E_\Phi$. This is not true in the category of semi-algebraic spaces, though: the main theorem of \cite{schei}, states the equivalence between the existence of the geometric quotient $\R^M/E_\Phi$ as a semi-algebraic space and that of a semi-algebraic subset of the space of parameters with a certain property.  Making use of this equivalence, in Section \ref{sec:shallow-networks} we provide a counterexample to it in this setting, thus
\begin{thmA}\label{th:sheiderer-introduction}
    In general, the neuromanifold $\mathcal M_\d$ of a ReLU neural network is not a semi-algebraic quotient of the space of parameters by the equivalence relation induced by the parametrization map.
\end{thmA}

In Section \ref{sec:identifiable-symmetries} we approach the problem of identifiability from an algebro-geometric point of view. We introduce the notion of \emph{identifiable}  and \emph{strongly identifiable} open subsets of the space of parameters, i.e.\ open subsets over which the fibers of the parametrization are trivial or, in other words, open subsets of $\R^M$ over which the network does not show any \emph{hidden} symmetry behavior. The introduction of identifiable open subsets of the space of parameters is an algebro-geometric framework to study more in depth similar questions as those studied in  \cite{GriLinRol}. Using methods from real algebraic geometry, in particular the Tarski-Seidenberg theorem (see for example \cite{realalgebraicgeometry}), we prove the following:
\begin{thmB}\label{conj:intro}
    Given a ReLU neural network of architecture $\mathbf{d}$, the maximal identifiable open (respectively maximal strongly identifiable open) subset of $\R^M$  is  semi-algebraic.
\end{thmB}
As we notice that there is nothing special about ReLU except its semi-algebraicity, with the appropriate notion of identifiability and an analogous proof we obtain the following more general statement. 
\begin{corC}\label{cor:intro}
    Given a neural network of architecture $\mathbf{d}$ whose activation function at each layer is semi-algebraic, the maximal identifiable open (respectively maximal strongly identifiable open) subset of $\R^M$  is  semi-algebraic. As a consequence, the subset of  \emph{identifiable} functions  $\mathcal{M}_{\d}^{\textup{Id}}$ in the neuromanifold has a canonical semi-algebraic structure.
\end{corC}
Any piecewise polynomial function, piecewise rational function and many other types of common activation functions, with or without biases, satisfy the hypothesis of Corollary \ref{cor:intro}, generalizing vastly the field of application of Theorem \ref{conj:intro}.

Lastly, in Section \ref{sec:shallow-networks} we restrict to the case of shallow ReLU networks, i.e.\ when $L=1$. Here we provide the promised counterexample to the semi-algebraicity of the neuromanifold as a quotient (Example \ref{ex:not-schei}), and describe the maximal identifiable subset explicitly, strengthening results from \cite{ramakrishnan}. Interestingly, we prove that for input dimension $\ge 3$ and  width $\ge 6$, the identifiable subset is not euclidean open (thus not Zariski open), proving that identifiability is not an open condition in the space of parameters.

\subsection*{Acknowledgements}
We thank Kathlén Kohn for bringing this topic to our attention, Antonio Lerario for sharing his ideas about the content of Section \ref{sec:semi-algebraicity} during a visit to KTH, and Elisenda Grigsby and Giovanni Marchetti for several useful conversations. We also thank Yaohua Liu for pointing out a mistake in a previous version of this paper. Part of this work was carried out while S.M. was partially supported by the Wallenberg AI, Autonomous Systems and Software Program (WASP) funded by the Knut and Alice Wallenberg Foundation.  A.F. also acknowledges support from the Wallenberg AI, Autonomous Systems and Software Program (WASP) funded by the Knut and Alice Wallenberg Foundation.

\section{Pointwise and abstract semi-algebraicity}\label{sec:semi-algebraicity}

As we already mentioned in the introduction, endowing the neuromanifold with a geometric structure could help immensely when dealing with concrete optimization problems related to minimizing the loss function. For a semi-algebraic activation function, one may wonder if the neuromanifold has a natural semi-algebraic structure. In the case of ReLU, one major difficulty relies in the (still very mysterious) semi-algebraic structure of the space of continuous functions. In what follows, we analyze different ways to think of the semi-algebraicity of a neuromanifold. 

In \cite{montufarbrand}, a way of realizing $\mathcal{M}_{\d}$ as the image of a semi-algebraic space is described. First, one uses the results of \cite{zhang2018tropical} to argue that every $f \in \mathcal{M}_{\d}$ can be written as a difference $f=g-h$ of two \emph{convex} piecewise linear functions (i.e.\ as a tropical rational function) with at most $m(\d)$ and $n(\d)$ linear regions, respectively. Denoting the space of parameters of such functions by $\Theta(d,m,n)$,  \cite[Theorem 3.2]{montufarbrand} states that $\mathcal{M}_{\d}$ can be attained as the image of a (non-unique) semi-algebraic subset $\mathcal S$ of $\Theta(d,m,n)$. This semi-algebraic subset depends on the particular way one decides to write a piecewise linear functions as a difference of two convex ones. Furthermore, the parametrization map restricted to $\mathcal S$ still has redundancies, i.e.\ it is not injective. 

While this is a way of understanding $\mathcal{M}_{\d}$ as the image of a semi-algebraic subset of a finite dimensional space, it is different from our goals in this section. In particular it is not known when the decomposition $f=g-h$ is unique \cite{scholtes1992minimal,urbanski1996minimal,grzybowski1994minimal}, thus how to find (if possible) a semi-algebraic subset of a finite dimensional space which is \emph{isomorphic} to $\mathcal M _{\d}$. We instead wish to consider $\mathcal{M}_{\d}$ as as a subset of $\calC^0(\R^{d_0})$  directly, and in particular seek to find semi-algebraic structures that respect the natural topology of $\calC^0(\R^{d_0})$. 

The first natural approach is to endow the space of continuous function with some infinite-dimensional semi-algebraic structure, so that one can at least show that there is some compatibility between the semi-algebraic structure on the space of parameters and that of the space of functions through the parametrization map. We want to remark that it is possible that some of the claims made in this section can be generalized to other activation functions using the notions of tame topologies and o-minimal structures.

For the rest of this subsection, we set $d:=d_0$ the width of the first layer.

 \begin{definition}
     A continuous map $C^0(\R^d)\rightarrow \R$ is said to be \emph{polynomial} if for every finite dimensional vector subspace $V\subset C^0(\R^d)$ the restriction $V\rightarrow \R$ is polynomial (in the usual sense). We can say that a continuous function $\cM \rightarrow C^0(\R^d)$ from a semi-algebraic space $\cM$ is \emph{semi-algebraic} if for every polynomial morphism $C^0(\R^d)\rightarrow \mathbb{R}$ the composition 
    $$\cM \rightarrow C^0(\R^d) \rightarrow \R$$
    is semi-algebraic (in the usual sense). 
\end{definition}
Although the previous one is the most natural definition, we introduce a second (seemingly weaker) definition, which is easier to test.

\begin{definition}
     Let $\cM$ be a semi-algebraic space  and suppose that we have a continuous morphism $\alpha:\cM \rightarrow C^0(\R^d)$ for some positive integer $d$.
    We say that $\alpha$ is \emph{pointwise semi-algebraic} if for every $x \in \R^d$ the composition
    \[
        \cM \stackrel{\alpha}{\rightarrow} C^0(\R^d) \stackrel{\pi_x}{\rightarrow} \R
    \]
    is semi-algebraic, where the second morphism is the projection $\pi_x$ defined by the association $g \mapsto g(x)$ for every $g \in C^0(\R^d)$.
\end{definition}

\begin{remark}
    It is clear that every semi-algebraic morphism $\cM \rightarrow C^0(\R^d)$ is pointwise semi-algebraic.
\end{remark}

Checking that the parametrization morphism is pointwise semi-algebraic follows from the following more general fact.

\begin{lemma}\label{lem:mega-para-is-semi-alg}
    The map $F_{\mathbf{d}}:\R^M \times \R^{d_0} \rightarrow \R^{d_L}$ as in equation (\ref{eq:network}) is semi-algebraic for any network with semi-algebraic activation functions (in particular for ReLU) and for any architecture $\mathbf d = (d_0, \dots , d_L)$. 
\end{lemma}
\begin{proof}
    The pre-activation affine map at each layer is polynomial in the parameters, biases and activation functions of the previous layer, thus semi-algebraic.
    Since the activation function $\sigma$ is semi-algebraic, its coordinatewise application is also semi-algebraic. The result follows as iterative composition of semi-algebraic functions is semi-algebraic.
\end{proof}
\begin{lemma}\label{lem:para-is-point-alg}
    The parametrization map $\Phi:\R^M \rightarrow C^0(\R^{d_0})$ associated to a neural network with semi-algebraic activation functions is pointwise semi-algebraic. 
\end{lemma}

\begin{proof}
    Let $F_{\mathbf{d}}:\R^M \times \R^{d_0} \rightarrow \R$ be the function as in Lemma \ref{lem:mega-para-is-semi-alg} in the special case when $d_L=1$. Then for every $x \in \R^{d_0}$ the composition 
    \[
        \R^M \stackrel{\Phi}{\rightarrow} C^0(\R^d_0) \stackrel{\pi_x}{\rightarrow} \R
    \]
    is equal to the function $F_{\mathbf{d}}(-,x):\R^M  \rightarrow \R$, which is semi-algebraic because $F_{\mathbf{d}}$ is.
\end{proof}

We claim that pointwise semi-algebraicity may help in endowing the parametrization morphism (and thus the neuromanifold) with a pro-semi-algebraic structure. This is not as strong as a semi-algebraic one, but we argue it could still be useful for optimization arguments. 

Let $\QQ$ be the countable dense subset of rational numbers in $\R$ and let $\iota:\N \rightarrow \QQ^d$ be a bijective map of sets. Consider the limit diagram
$$  \dots \rightarrow  \R^n \rightarrow \R^{n-1} \rightarrow  \dots \rightarrow \R$$
in the category of topological vector spaces and denote by $\R^{\N}$ the limit (which is the product space with the product topology). Morphisms in the diagram above are just projections forgetting the last coordinate. Moreover, for every $n$ we have a morphism $\pi_n:C^0(\R^d) \rightarrow \R^n$ defined by the assignment $f\mapsto \big(f(\iota(1)),\dots, f(\iota(n))\big)$ and taking the limit we get a morphism $\pi: C^0(\R^d) \rightarrow \R^{\N}$ which is injective, since two continuous functions which agree on a dense subset agree on the whole space. Notice that the morphism $\pi$ is the composition of the morphism $C^0(\R^d) \rightarrow \R^{\R^d}$ with the natural projection morphism, where $\R^{\R^d}$ is the set of all functions $\R^d\rightarrow \R$ with the pointwise convergence topology.

This alternative point of view allows us to endow the parametrization morphism (and thus the neuromanifold) with a pro-semi-algebraic structure, namely it is the limit of the following diagram
\[
\begin{tikzcd}
                &                  & \R^M \arrow[d, "\Phi_{n-1}"] \arrow[ld, "\Phi_n"'] \arrow[rrd, "\Phi_1"] &                 &    \\
\dots \arrow[r] & \R^{n} \arrow[r] & \R^{n-1} \arrow[r]                                                       & \dots \arrow[r] & \R
\end{tikzcd}
\]
where $\Phi_n:=\pi_n \circ \Phi$. Hence, the neuromanifold (endowed with the pointwise convergence topology) can be described as the limit of the (affine) semi-algebraic spaces obtained as images of $\Phi_n$ when $n$ varies, thus it is a pro-object in the category of semi-algebraic spaces.

\begin{remark}
A priori, it is not clear whether the neuromanifold is representable by a semi-algebraic space. In algebraic geometry, one can prove the representability of limit diagrams assuming the affineness of the transition morphisms. All the results the authors are aware of in semi-algebraic geometry rely on some uniform boundedness of the presentation of this semi-algebraic spaces, since, by definition, semi-algebraic spaces are locally semi-algebraic subsets of a finite-dimensional vector space (over $\R$). It would be interesting to develop a semi-algebraic theory for Banach spaces, where these objects seem to naturally live. Nevertheless, by denoting as $\mathcal{A}_n$ the image of $\Phi_n$ (which is a semi-algebraic set thanks to \Cref{lem:para-is-point-alg}), one may argue that $\mathcal{A}_n$ \emph{approximates} the neuromanifold arbitrarily well as we let $n$ grow. Moreover, there could be a smart choice for the function $\iota$ that makes the approximation faster, depending on the architecture of the neural network. 

It is worth noticing that for the purpose of concrete computations (for instance in optimization problems), mathematicians (and computers) work with finite amount of data (for instance to estimate the $L^{\infty}$ norm, or in general any loss function), thus it seems natural to study the semi-algebraic sets $\mathcal{A}_n$ as computationally tractable approximations of the neuromanifold. Indeed, there are already instances in the literature: in \cite{AlexMontu}, the authors study exactly the geometry of (the Zariski closure of) $\mathcal{A}_n$. Furthermore, although the pointwise convergence topology is coarser than the compact-open topology, for the sake of computations they actually agree, specifically when checking if two functions agree on a finite data set.
\end{remark}

\begin{remark}
 As already pointed out in the previous remark, the algebraic variety $\overline{\cA_n}^{\text{Zar}}$ already appears in the literature: see \cite{AlexMontu}. Therefore one may be tempted to study the pro-algebraic variety obtained by taking the limit of $\overline{\cA_n}^{\text{Zar}} \subseteq \R^n$. Already at the finite-dimensional level though, it is clear that passing to the Zariski closure loses a lot of information: clearly, the Zariski closure of any ball in $\R^n$ (no matter how small) is the whole space. Even worse, the limit procedure does not necessarily behave
 well with taking Zariski closures.
\end{remark}

Another natural approach is to establish the semi-algebraicity of the neuromanifold abstractly, namely without forcing the image inside the space of continuous functions to be semi-algebraic (in any sense). This is usually done using the language of groupoids. In what follows, we will make parallels with the theory of groupoids (more specifically equivalence relations) in the category of schemes, which brought to the birth of algebraic spaces.

Let us start with a very simple observation for topological spaces.

\begin{remark}
    Let $f:X\rightarrow Y$ be a continuous map of topological spaces. We can form the equivalence relation on $X$ defined by the subset $E_f:=X\times_{Y,f} X \subset X \times X$, namely the pairs $(x_1,x_2) \in X\times X$ such that $f(x_1)=f(x_2)$. If $f$ is a strict (i.e.\ a subset $U$ of $Y$ is open if and only if $f^{-1}(U)$ is open in $X$) and surjective map, then $Y$ is the quotient in the category of topological spaces of $X$ by $E_f$. 
\end{remark}

Inspired by the previous remark, one can try to come up with conditions on a general equivalence relation $E\subset X\times X$ which imply the existence of the quotient. This is in general very difficult to do, and depending on the geometric structure chosen, there could be different answers. In algebraic geometry, if $X$ is a scheme, we know that flat and finitely presented groupoids give rise to algebraic spaces (or in general algebraic stacks if the morphism $E\rightarrow X\times X$ is not injective). Surprisingly enough, in the semi-algebraic setting the following theorem provides a complete answer to the existence of a quotient in the category of semi-algebraic spaces. 

\begin{theorem}{(\cite[Theorem 5.1]{schei})}\label{theo:schie}
    Let $E \subset \R^M \times \R^M$ be a closed semi-algebraic equivalence relation of $\R^M$ and let $p_1$ and $p_2$ be the two projections $E \rightarrow \R^M$. The following are equivalent:

    (i) the geometric quotient $\R^M/E$ exists in the category of affine semi-algebraic spaces;

    (ii) there exists a closed semi-algebraic subset $K\subset \R^M$ such that the morphism 
        \[
            p_2\vert_{p_1^{-1}(K)}: p_1^{-1}(K) \rightarrow  \R^M
        \]
    is proper and surjective.
\end{theorem}

Let $E_{\Phi}\subset \R^M \times \R^M$ be the equivalence relation 
\[
    E_{\Phi} := \{(v,w) \in \R^M \times \R^M \mid \forall x \in \R^{d_0} \, f_{v}(x) = f_{w}(x)\}.
\]
It is a direct consequence of Lemma \ref{lem:mega-para-is-semi-alg} and of the Tarski-Seidenberg theorem that:
\begin{lemma} \label{lem:semialg-eq-rel}
The equivalence relation $E_{\Phi}$ is a semi-algebraic subset of $\R^M \times \R^M$.
\end{lemma}

Moreover, we have the following:

\begin{lemma}\label{lem:topology}
Let $f:X\rightarrow Y$ be a continuous map of topological spaces and let $E_f:=X\times_{Y,f} X \subset X \times X$ be the equivalence relation on X given by $f$, as above. Suppose that $X$ and $E_f\subset X\times X$ are affine semi-algebraic spaces, and that the quotient $X/E_f$ exists as an affine semi-algebraic space. Then there exists a canonical morphism $\pi:X/E_f \rightarrow Y$ of topological spaces which is bijective. Moreover, if $f$ is strict and surjective, then $\pi$ is an homeomorphism. 
\end{lemma}

\begin{remark}
    It is clear that if the geometric quotient exists as an affine semi-algebraic space $\mathcal M$, then by \Cref{lem:topology} we get that the morphism $\cM \hookrightarrow C^0(\R^d)$ is pointwise semi-algebraic. 
\end{remark}

Suppose that $\Phi:\R^M \rightarrow \im \Phi$ is strict and that $\cM:=\R^M/E_{\Phi}$ exists as an affine semi-algebraic space. Then we could identify topologically $\im \Phi$ with $\cM$, giving the neuromanifold a semi-algebraic structure. Unfortunately, the assumptions above fail even in the case of a shallow neural network, as we show later in  \Cref{ex:not-schei}.

\section{Identifiable subsets of parameters and hidden symmetries}\label{sec:identifiable-symmetries}

Even though we cannot endow the neuromanifold with a semi-algebraic structure of a quotient, in this section we want to convince the reader that the study of the neuromanifold as a groupoid in semi-algebraic space can lead to some interesting remarks regarding the \emph{singularities} of the neuromanifold. These singularities are somehow reflected in the fibers of the parametrization morphism $\Phi$, which can be studied forgetting about the morphism itself and focusing on the groupoid $E_{\Phi}\hookrightarrow \R^M \times \R^M$. Our intent in what follows is to absorb the generic symmetries of a ReLU neural network without changing the image of the morphism $\Phi$, i.e. the neuromanifold. We will focus on the case with no bias, for simplicity.

Let $\mathbf{d}:=(d_0,\dots,d_{L})\in \N^{L+1}$ be the architecture vector of length $L$ as above. Recall that we set $d_{L}=1$ and that the space of parameters
\begin{align*}
    \R^M & \simeq \textup{Mat}_{d_1,d_0}(\R) \times \textup{Mat}_{d_2,d_1}(\R) \times \dots \times \textup{Mat}_{1,d_{L-1}}(\R) \\  & \simeq (\R^{d_0})^{d_1} \times (\R^{d_1})^{d_1} \times \dots \times \R^{d_{L-1}}
\end{align*}
is an affine space of dimension $M = M(\mathbf{d})=\sum_{i=0}^{L-1}d_id_{i+1}$ in this case. Given an architecture, there exist two type of well-known symmetries (see \cite{ArmJod} for a more detailed discussion):
\begin{itemize}
    \item[(s)] scaling: for any neuron in any hidden layer, multiplying the incoming parameters and the bias by any positive constant, while dividing the outgoing parameters by the same constant; 
    \item[(p)] permutation: reordering the neurons in any hidden layer, along with the corresponding permutations of the parameters and biases associated with them.
\end{itemize}

Formally, the group of these symmetries can be described as follows. Consider the group $\GL_{\mathbf{d}}:=\prod_{i=0}^{L} \GL_{d_i}(\R)$: we can define an action on $\textup{Mat}_{\mathbf{d}}:=\prod_{i=0}^{L-1} \textup{Mat}_{d_{i+1},d_{i}}(\R)$ by the association 
\[
    (A_0,\dots, A_{L}) \cdot  (M_0, M_1, \dots, M_L) := (A_1 M_0 A_0^{-1},\dots , A_{L} M_{L-1} A_{L-1}^{-1}) 
\]
for $ A_i \in \GL_{d_i}(\R) $ and $M_i \in \textup{Mat}_{d_{i+1},d_i}(\R)$. 

From now on, we will denote by $\R_+$ the multiplicative subgroup of $\R\setminus \{0\}$ given by the positive real numbers. Furthermore, we will denote by $\Sigma_i$ the symmetric group on $d_i$-elements. 
\begin{definition}
    We denote by ${\rm Sc}(\mathbf{d})$ the \emph{scaling group} $(\R_+)^L$ seen as the subgroup of $\GL_{\mathbf{d}}$ where we identify the $i$-th copy of $\R_+$ inside $\GL_{d_i}(\R)$ as the positive scalar matrices (notice that for $i=0$ and for $i=L$ we are setting the projection map ${\rm Sc}(\mathbf{d}) \rightarrow \GL_{d_i}(\R)$ to be trivial).

    Moreover, we denote by ${\rm Pr}(\mathbf{d})$ the \emph{permutation group} $\prod_{i=1}^{L}\Sigma_{d_i}$ seen as the subgroup of $\GL_{\mathbf{d}}$ where we identify $\Sigma_{d_i}$ inside $\GL_{d_i}(\R)$ as the permutation matrices (as above, for $i=0$ and for $i=L$ we are setting the projection map ${\rm Pr}(\mathbf{d}) \rightarrow \GL_{d_i}(\R)$ to be trivial).

    Finally, we denote by $G(\mathbf{d})$ the subgroup generated by ${\rm Sc}(\mathbf{d})$ and ${\rm Pr}(\mathbf{d})$ inside $\GL_{\mathbf{d}}$, i.e.\ $G(\mathbf{d}):= {\rm Sc}(\mathbf{d}) \rtimes {\rm Pr}(\mathbf{d})$. The action of $G(\mathbf{d})$ is determined by the action of $\GL_{\mathbf{d}}$ on $\textup{Mat}_{\mathbf{d}}$ described above. 
\end{definition}

    We introduce now the notion of (strongly) identifiable parameters and subsets and we work towards proving that the subset of (strongly) identifiable parameters for a ReLU neural network is semi-algebraic. The following general definition formalizes statements from \cite{phuong2020functional}, \cite{RolKor} and \cite{GriLinRol}. 
\begin{definition}\label{def:identifiability}
    Let $G$ be a group acting on a set $X$ and $f \co X \rightarrow Y$ be a $G$-invariant map. A point $x \in X$ is \emph{identifiable} if $f^{-1}(f(x))=G\cdot x$, and \emph{strongly identifiable} if it is identifiable and the orbit has trivial stabilizers.

    We denote as $\cI_f$ and $s\cI_f$ the subsets of identifiable and strongly identifiable points of $X$, respectively. Notice that they are $G$-invariant by construction.
\end{definition}
Identifiability is a property of points in the space of parameters. We can also of consider identifiable \emph{subsets} of parameters, i.e.\ subsets of that are identifiable relatively to themselves, as opposed to parameters that are identifiable \emph{per se}. We call the former sets \emph{weakly identifiable} as in the following definition.

\begin{definition}\label{def:identifiable-sets}
    A $G$-invariant subset $U \subseteq X$ is:
    \begin{itemize}
        \item \emph{weakly identifiable} if $f^{-1}(f(x))\cap U=G\cdot x$ for all $x \in U$;
        \item \emph{identifiable} if each point $x \in U$ is identifiable, i.e. $U \subseteq \cI_f$;
        \item \emph{strongly identifiable} if each point $x \in U$ is strongly identifiable, i.e.\ $U \subseteq s\cI_{f}$.
    \end{itemize}
    Notice that $U$ is identifiable if and only if it is weakly identifiable and saturated with respect to $f$, namely $U=f^{-1}f(U)$.
\end{definition}
\begin{remark}  
    Let us specify what has been proven in the literature (at least in some form) regarding identifiable subsets: in \cite{RolKor} it is proven that, under the \emph{Linear Region Assumption}, there exists a set whose complement has measure zero and such that it is possible to reconstruct the neural network from the function up to scaling and permutations. Moreover, in \cite{GriLinRol}, the authors prove that there always exists a set of positive measure on which the only symmetries are given by scaling and permutations.  It is not clear to us if this positive measure set is a semi-algebraic open. Notably, though, it is proven that this set contains at least an open (showing as a byproduct that the Linear Region Assumption is an open condition in the space of parameters), thus their result implies that there always exists a weakly identifiable semi-algebraic open (since an open ball is semi-algebraic). Furthermore in \cite{ramakrishnan} the authore classifies all the symmetries of shallow ReLU networks, and in \cite{gegenfurtner2026symmetries} the authors classify all the symmetries of ReLU networks with two hidden layers and a bottleneck. 
\end{remark}

    Our first goal is to absorb the scaling group action. We show that we can consider an open of $\R^M$, without changing the image of $\Phi$. A similar idea has been used in \cite{AroLiLyu}. Let $P_{\mathbf{d}} \subset \R^M$ be the open subset defined as the product 
    \[
        (\R^{d_0}\setminus 0)^{d_1} \times (\R^{d_1}\setminus 0)^{d_2} \times \dots \times (\R^{d_{L-2}}\setminus 0)^{d_{L-1}} \times \R^{d_{L-1}}.
    \]
Notice that the last factor of the product remains untouched. 
\begin{lemma}\label{lem:take-away-zero}
    With the notation above, $\Phi(P_{\mathbf{d}})=\Phi(\R^M)\subseteq C^0(\R^{d_0})$.
\end{lemma}

\begin{proof}
The idea is essentially that we can move any zero-rows in the original parametrization to be in the final layer.
    Define 
        \[
            P_j:= \prod_{i=0}^{j-1} (\R^{d_i} \setminus 0)^{d_{i+1}} \times \prod_{i=j}^{L} (\R^{d_i})^{d_{i+1}}
        \]
    where $P_{L}=P_{\mathbf{d}}$ and $P_{j+1}\subset P_{j}$: we will prove that for every $0\leq j\leq L-1$ we have $\Phi(P_j)=\Phi(P_{j+1})$. This clearly implies the statement since 
    \[
        \Phi(P_{\mathbf{d}})=\Phi(P_{L})=\Phi(P_{L-1})=\dots=\Phi(P_{1})=\Phi(P_{0})=\Phi(\R^M).
    \]
    Let $0\leq j\leq L-1$ and $(A_0,\dots,A_L) \in P_{j}\setminus P_{j+1}$. Therefore (at least) one of the rows of the matrix $A_{j}\in (\R^{d_j})^{d_{j+1}}$ is the origin of $\R^{d_j}$. Suppose without loss of generality that the first $k$ rows of $A_j$ are the zero vector for $k \geq 1$. Then we can set
    \begin{itemize}
        \item $A_i'=A_i$ for every $i\notin \{j,j+1\}$;
        \item $A_j'$ to be the matrix equal to $A_j$ except for the first $k$ rows, where we can choose any $k$-tuple of non-zero vectors;
        \item $A'_{j+1}$ to be the matrix equal to $A_{j+1}$ except for the first $k$ columns, which we set to be the zero vector.
    \end{itemize} 
    It is clear now that $(A_1',\dots,A_L') \in P_{j+1}$ and $\Phi(A_0',\dots,A_L')=\Phi(A_0,\dots,A_L)$.
\end{proof}
Notice that $P_{\mathbf{d}}$ is $G(\mathbf{d})$-invariant, thus the action restricts to $P_{\mathbf{d}}$. Moreover, the action of ${\rm Sc}(\mathbf{d})$ is free on $P_{\mathbf{d}}$, and in fact 
\[
    P_{\mathbf{d}}/{\rm Sc}(\mathbf{d}) \simeq \R^{d_L} \times \prod_{i=0}^{L-1} (S_i)^{d_{i+1}}
\]
where we denote by $S_i$ the $(d_i-1)$-dimensional sphere. To ease the notation, we set $\cP:= P_{\mathbf{d}}/{\rm Sc}(\mathbf{d})$. Therefore, since the restriction of $\Phi$ to $P_{\mathbf{d}}$ is $G(\mathbf{d})$-invariant, the neuromanifold is determined by the morphism $\varphi: \cP \rightarrow C^0(\R^{d_0})$ obtained by taking the quotient by ${\rm Sc}(\d)$, i.e.\ the restriction of $\Phi$ factors as follows:

\[
  \begin{tikzcd}
   P_{\mathbf{d}} \arrow[rr, two heads, "\Phi \mid_{P_{\mathbf{d}} }"] \arrow[rd, two heads] &     & \mathcal{M}_{\mathbf{d} }  \\ & \cP  \arrow[ru, two heads, swap , "\varphi"]&.  
   \end{tikzcd}
\]

In other words, we reduced ourselves to study the equivalence relation $E_{\varphi}\subset \cP \times \cP$. Notice that ${\Pr}(\mathbf{d})$ still acts on $\cP$ non-trivially: we get a morphism 
\[
    {\Pr}(\mathbf{d}) \times \cP \longrightarrow \cP \times \cP
\]
induced by the projection on the first factor and by the action on the second factor of the product. Since the morphism $\varphi$ is ${\Pr}(\mathbf{d})$-invariant, the morphism above actually factors through $E_{\varphi}\subset \cP \times \cP$.

We are ready to prove that the subset of (strongly) identifiable parameters in $\mathcal P$ is semi-algebraic:

\begin{theorem}\label{th:max-open-in-quotient}
    For a ReLU neural network of fixed architecture $\mathbf{d}$, the subset $\cI_{\varphi} \subseteq \cP$ of identifiable parameters is semi-algebraic. In particular, the interior of $\cI_{\varphi}$ in $\cP$ is the unique maximal identifiable open subset of $\cP$, and it is semi-algebraic. The same result holds for the subset $s\cI_{\varphi}$ of strongly identifiable parameters.
\end{theorem}

\begin{proof}
Recall that
\begin{align*}
    \cI_{\varphi} &:=\{w \in \mathcal P \mid \varphi^{-1}(\varphi(w)) = \Pr(\mathbf d) \cdot w\},  \\
    \mathcal F &:= \{w \in \mathcal P \mid \textup{Stab}(w) = \{\rm id\}\}, \\
    s\cI_{\varphi} &:= \cI_{\varphi} \cap \mathcal F.
\end{align*}
By definition, given a $\Pr(\mathbf{d})$-invariant $U \subseteq \mathcal P$ the following equivalences hold: 
\begin{align*}
   U \text{ is identifiable} 
 \iff & U \subseteq \cI_{\varphi} \\
 U \text{ is strongly identifiable} 
 \iff & U \subseteq s\cI_{\varphi}
\end{align*}
 From Lemma \ref{lem:semialg-eq-rel} we have that $E_\varphi = E_\Phi \cap (\mathcal P \times \mathcal P)$ is a semi-algebraic subset of $\mathcal P \times \mathcal P$ as it is the intersection of two semi-algebraic subsets. For every $\tau \in \Pr(\mathbf d)$, since $\mathcal P$ is semialgebraic and $\tau$ acts linearly, the graph $\Gamma_\tau = \{(w, \tau \cdot w) \mid w \in \mathcal P\} \subseteq \mathcal P \times \mathcal P$ is semi-algebraic. Define $\mathcal R_{\Pr(\mathbf{d})} := \bigcup_{\tau \in \Pr(\mathbf d)} \Gamma_\tau $, then by Tarski-Seidenberg  
\[
    \mathcal P \setminus \cI_{\varphi} = \pi_1\left (E_{\varphi} \setminus \mathcal R_{\Pr(\mathbf{d})} \right )
\]
is semi-algebraic, hence $\cI_{\varphi}$ is semi-algebraic. Moreover
\[
    \mathcal F = \mathcal P \setminus \bigcup_{\tau \neq \textup{id}} \{w \in \mathcal P \mid \tau \cdot w = w\}
\]
is semi-algebraic as each of the sets $\{w \in \mathcal P \mid \tau \cdot w = w\}$ is a linear subspace (intersected with $\mathcal P$). Thus $s\cI_{\varphi}$ is semi-algebraic. 

We are just left to prove that $\textup{int}_{\mathcal P}(\cI_{\varphi})$ and  $\textup{int}_{\mathcal P}(s\cI_{\varphi})$ are semi-algebraic. This follows directly from the semi-algebraicity of $\cP$.
\end{proof}
We go back to considering the parametrization map $\Phi \co \R^{M} \rightarrow \cC(\R^{d_0})$, which is invariant with respect to the group $G(\mathbf d)$ of trivial symmetries. To prove the analogous result to Theorem \ref{th:max-open-in-quotient} for $\Phi$ we need the following lemma.

\begin{lemma}\label{lem:fiber-intersects-two-orbits}
    Let $w \in \R^M$ be an element such that $\Phi^{-1}(\Phi(w)) \cap (\R^M \setminus P_{\mathbf d}) \neq \emptyset$. Then $\Phi^{-1}(\Phi(w)) \cap P_{\mathbf d}$ consists of at least two distinct orbits of the $G(\mathbf{d})$-action.
\end{lemma}
\begin{proof}
We start by noticing that Lemma \ref{lem:take-away-zero} implies that $\Phi^{-1}(\Phi(w)) \cap P_{\mathbf d} \neq \emptyset$. As explained in the proof of the aforementioned lemma, we can choose two parameters $v,v' \in P_{\mathbf{d}}$ in a way that $\Phi(w) = \Phi(v)= \Phi(v')$ but $v \notin G(\mathbf{d}) \cdot v'$.  
\end{proof}

\begin{theorem}\label{th:max-open}
    For a ReLU neural network of fixed architecture $\mathbf{d}$, the subset $\cI_{\Phi} \subseteq \R^M$ of identifiable parameters is semi-algebraic. In particular, the interior of $\cI_{\Phi}$ is the unique maximal identifiable open subset of $\R^M$, and it is semi-algebraic. The same result holds for the subset $s\cI_{\Phi}$ of strongly identifiable parameters.
\end{theorem}
\begin{proof}
    It follows almost immediately from the proof of Lemma \ref{lem:take-away-zero} that any point $w \in \R^M \setminus P(\mathbf{d})$ is not identifiable, i.e.\ it lies in $\R^M \setminus \cI_\Phi$. Indeed, suppose $w \notin P(\mathbf{d})$ and, as in the proof of Lemma \ref{lem:take-away-zero}, construct some element $w' \in P(\mathbf{d})$ such that $\Phi(w)=\Phi(w')$. As $P(\mathbf{d})$ is $G(\mathbf{d})$-invariant, $w'\notin G(\mathbf{d}) \cdot w$, thus $w$ is not identifiable. From this it follows that $\cI_\Phi \subseteq P_{\mathbf{d}}$.

    Let $q \co P_{\mathbf{d}} \twoheadrightarrow \cP$ be the quotient map. We claim that $\cI_\Phi = q^{-1}(\cI_\varphi)$ and $s\cI_\Phi = q^{-1}(s\cI_\varphi)$.
    
    We start by proving that $\cI_\Phi \subseteq q^{-1}(\cI_\varphi)$. Let $w \in \cI_\Phi$ and $u \in \cP$ be elements such that $\varphi(u) = \varphi(q(w))$. Then we can choose an element $v \in P_{\mathbf{d}}$ such that $u = q(v)$ and it follows that $\Phi(v)=\Phi(w)$. As $w$ is identifiable, $v \in G(\mathbf d) \cdot w$. Therefore $u=q(v) \in \Pr(\mathbf d) \cdot q(w)$. Thus $q(w) \in \cI_\varphi$, i.e.\ $w \in q^{-1}(\cI_\varphi)$.

    To prove that $ q^{-1}(\cI_\varphi) \subseteq \cI_\Phi$, suppose $w \in  q^{-1}(\cI_\varphi)$ such that $\Phi(w) = \Phi(v)$ for an element $v \in \R^M$. If $v \in P_{\mathbf{d}}$, we have that $q(v) \in \Pr(\mathbf d) \cdot q(w)$ thus, since $q$ is the quotient by positive scalar multiplication, $v \in G(\mathbf d) \cdot w$. If $v \notin P_{\mathbf{d}}$, by Lemma \ref{lem:fiber-intersects-two-orbits} we have that $\Phi^{-1}(\Phi(w))$ consists of at least two different orbits for the action of $G(\mathbf d)$. It follows that $q(w)$ is not identifiable in $\cP$, which is a contradiction. Thus $v \in P_{\mathbf d}$ and we have the equality $ q^{-1}(\cI_\varphi) = \cI_\Phi$.

    For what concerns $s\cI_\Phi$, as the action of scaling is free on $P_{\mathbf d}$, the orbit of $w \in P_{\mathbf d}$ has trivial stabilizers for the action of $G(\mathbf d)$ if and only if the orbit of $q(w)$ has trivial stabilizers for the action of $\Pr(\mathbf d)$. Hence $ q^{-1}(s\cI_\varphi) = s\cI_\Phi$.

    The quotient map $q \co P_{\mathbf{d}} \twoheadrightarrow \mathcal P$ is semi-algebraic as it is the composition of semi-algebraic maps given by iteratively rescaling the parameters. This implies that $q^{-1}(\cI_\varphi)$ and $q^{-1}(s\cI_\varphi)$ are semi-algebraic in $P_{\mathbf{d}}$. From the above discussion we have that $\cI_\Phi$ and $s\cI_\Phi$ are semi-algebraic in $\R^M$.
\end{proof}
\begin{remark}
    Since $q$ is open and $P_{\mathbf{d}}$ is open in $\R^M$ the following holds
    \[
        \textup{int}_{\R^M}(\cI_\Phi) = q^{-1}(\textup{int}_{\cP}(\cI_\varphi)) \quad \quad \quad \textup{int}_{\R^M}(s\cI_\Phi) = q^{-1}(\textup{int}_{\cP}(s\cI_\varphi)).
    \]
\end{remark}
Being (strongly) identifiable is essentially a property of functions in $\mathcal M_{\d}$.

\begin{definition}
    
    A function $f \in \mathcal M_{\d}$ is 
    \begin{itemize}
    \item\emph{identifiable} if $\varphi^{-1}(f) = \Pr(\d) \cdot w$ for some $w \in \cP$. Denote as $\mathcal M_{\d}^{\textup{Id}}$ the subset of identifiable functions in $\mathcal M_{\d}$;
    \item  \emph{strongly identifiable} if it is identifiable and $\mathrm{Stab}_{\Pr(\d)}(w)=\{\textup{id}\}$. Denote as $\mathcal M_{\d}^{\textup{sId}}$ the subset of strongly identifiable functions in $\mathcal M_{\d}$.
    \end{itemize}
\end{definition}
The following corollary is an immediate consequence of Theorem \ref{th:max-open-in-quotient}. 
\begin{corollary}
    The subsets $\mathcal M_{\d}^{\textup{Id}}$ of identifiable functions and $\mathcal M_{\d}^{\textup{sId}}$ of strongly identifiable functions inside $\mathcal M_{\d}$ are canonically endowed with a finite dimensional semi-algebraic structure.
\end{corollary}
\begin{proof}
    From Theorem \ref{th:max-open-in-quotient} we know that both $\calI_\varphi$ and $s\calI_{\varphi}$ are semi-algebraic. The finite group $\Pr(\d)$ is semi-algebraic and it acts semi-algebraically on the two aforementioned sets, thus $\calI_\varphi/\Pr(\d)$ and $s\calI_{\varphi}/\Pr(\d)$ are semi-algebraic. As $\mathcal M_{\d}^{\textup{Id}}$ and $\mathcal M_{\d}^{\textup{sId}}$ are in bijection with $\calI_\varphi/\Pr(\d)$  and $s\calI_\varphi/\Pr(\d)$, respectively,  via (the map induced by) $\varphi$, we conclude.
\end{proof}
\begin{remark}\label{rem:how-to-prove-general}
    In the previous pages we made use of the quotient $\cP$ to prove Theorem \ref{th:max-open} via  Theorem \ref{th:max-open-in-quotient}. We did this because some computations in the following will be clearer if performed in $\cP$, especially 
    in Section \ref{sec:shallow-networks}.

    Notice that we could have proven Theorem \ref{th:max-open} directly without making use of $\cP$ and Theorem \ref{th:max-open-in-quotient}. As $G(\mathbf{d})$ is not a finite group, the same proof as for Theorem \ref{th:max-open-in-quotient} does not apply verbatim: in particular when constructing the set $\mathcal{R}_{G(\mathbf{d})}$ as an infinite union of graphs we would not be able to conclude its semi-algebraicity. The following takes care of this issue.

    The group $\R_+$ and any permutation group $\Sigma_i$ are semi-algebraic. Thus $\textup{Sc}(\mathbf{d})$ and $\Pr(\mathbf{d})$ are semi-algebraic, as they are products of semi-algebraic groups. The same applies to $G(\mathbf{d})$. Similarly, its action on $\R^M$ is semi-algebraic. 
    
    Consider the set 
    \[
        \mathcal D := \{(g,w,w') \in G(\mathbf{d}) \times \R^M \times \R^M \mid w' = g \cdot w\}.
    \]
    As the action of $G(\mathbf{d})$ is semi-algebraic, the set $\mathcal D$ is semi-algebraic. Define 
    \[
        \mathcal{R}_{G(\mathbf{d})}:= \pi_{2,3}(\mathcal{D}) \subseteq \R^M \times \R^M.
    \]
    By Tarski-Seidenberg, it is semi-algebraic, and an analogous proof to that of Theorem $\ref{th:max-open-in-quotient}$ yields directly Theorem \ref{th:max-open}.
\end{remark}

We want to stress that to prove Theorem \ref{th:max-open-in-quotient} and Theorem \ref{th:max-open} we did not use any specific feature of ReLU neural networks. The only essential fact is that the ReLU activation function is semi-algebraic. Indeed, with analogous proof to that explained in Remark \ref{rem:how-to-prove-general}, one obtains the following more general result.
\begin{corollary}\label{cor:max-open}
    Consider a neural network of fixed architecture $\mathbf{d}$, whose activation function at each layer is semi-algebraic. Let $G$ be a semi-algebraic group acting semi-algebraically on $\R^M$, such that the parametrization map $\Phi \co \R^M \rightarrow \{f \co \R^{d_0} \rightarrow  \R^{d_L}\}$ is $G$-invariant. Then the subset $\cI_{\Phi} \subseteq \R^M$ of identifiable parameters is semi-algebraic. In particular, the interior of $\cI_{\Phi}$ is the unique maximal identifiable open subset of $\R^M$, and it is semi-algebraic. The same result holds for the subset $s\cI_{\Phi}$ of strongly identifiable parameters.
\end{corollary}

\begin{remark}
    Examples of other popular semi-algebraic activation functions are leaky ReLU, maxout, and in general any piecewise polynomial function, softsign, squareplus, ISRU, and rational activations. Notice that Corollary \ref{cor:max-open} also encompasses the case of non-zero biases. Conversely it does not apply to sigmoid, $\tanh$, GELU, and ELU, among others. 
\end{remark}

In the next section, we will study shallow networks proving in this setting that $\cI_{\varphi} = s\cI_{\varphi}$ as subsets of $\cP$, and giving an explicit presentation of this set. 

We conclude this section by showing that \Cref{th:max-open} does not hold for weakly identifiable open subsets. 

\subsection*{A counterexample to  \Cref{th:max-open} for weakly identifiable subsets}

Fix the architecture $\mathbf{d}=(2,2,1,1)$. Every function in the neuromanifold has the form
\begin{align*}
    f(x) = c\sigma(a\sigma(\langle u,x\rangle) -b \sigma(\langle v, x\rangle))
\end{align*}
for some $a,b,c\in \R$ and $u,v\in \R^2$. Let us restrict our attention to the following semi-algebraic set
\begin{align*}
    H = \{(u,v,a,b,c) \, \vert \, u,v \text{ linearly independent,} \, c > 0, \, ab>0 \}.
\end{align*}

On $H$, we can absorb the action of rescaling from $a,b,c$ into the ReLUs, and write
\begin{align*}
    f_{\bar{u},\bar{v}}(x) = \sigma(\sigma(\langle \bar{u},x\rangle)-\sigma(\langle \bar{v},x\rangle )).
\end{align*}
Formally, if $a,b>0$, $\bar{u}=cau$ and $\bar{v}=cbv$, and if $a,b<0$, $\bar{u}=-cbv$ and $\bar{v}=-cau$. Note that the map $\psi:H \to \R^2\times \R^2, (u,v,a,b,c)\mapsto (\bar{u},\bar{v})$ is semi-algebraic, and its fibers are $G(\mathrm{d})$-invariant. 

We set $Q:=\psi(H)$ and denote as $\phi$ the parametrization map restricted to $Q$. We start by constructing a hidden symmetry on $Q$: observe that the elementary identity 
\begin{align*}
    \sigma(\sigma(y)-\sigma(z))=\sigma(\sigma(y-z)-\sigma(-z)), \quad x,y\in \R.
\end{align*}
by setting $y=\sprod{\bar{u},x}$ and $z=\sprod{\bar{v},x}$ shows that 
\begin{align*}
    f_{\bar{u},\bar{v}} = f_{\bar{u}-\bar{v},-\bar{v}}.
\end{align*}
Define $\widetilde{T}(\bar{u},\bar{v})=(\bar{u}-\bar{v},-\bar{v})$. The function $\widetilde{T}$ is an involution on $Q$ that preserves the output function. Note that due to the linear independence of $\bar{u}$ and $\bar{v}$, the element $\widetilde{T}(\bar{u},\bar{v})$ is never in the orbit of $\bar{u},\bar{v}$, so that $\widetilde{T}$ defines a hidden symmetry.

Consider a point $(\bar{u}_0,\bar{v}_0)\in Q$ with  $\widetilde{T}(\bar{u}_0,\bar{v}_0)\neq (\bar{u}_0,\bar{v}_0)$, for instance $u_0=(2,0)$, $v_0=(0,1)$. Because of continuity of $\widetilde{T}$, there exists an open neighborhood $\tilde{\mathcal N}$ of $(\bar{u}_0,\bar{v}_0)$ with  $T(\widetilde{\mathcal N})\cap \widetilde{\mathcal N}=\emptyset$. 

The parametrization map is injective on $\widetilde{\calN}$. To see this, notice that since $\bar{u},\bar{v}$ are linearly independent, there exists linearly independent points $x_1$, $x_2$ with $\langle \bar{u},x_i\rangle >0>\langle \bar{v},x_i\rangle$. For such points and for $\bar{u},\bar{v}\in \widetilde{\calN}$ we get
\begin{align*}
    f_{\bar{u},\bar{v}}(x_i) = \sigma(\sigma(\langle \bar{u},x_i\rangle) - \sigma (\langle \bar{v},x_i\rangle)) = \sigma(\langle \bar{u},x_i\rangle) = \langle \bar{u},x_i\rangle.
\end{align*}
From these values, since the $x_i$ are linearly independent, we can determine the value of $\bar{u}$. Now,  by choosing linearly independent points $x_1$, $x_2$ with $\langle \bar{u},x_i\rangle >\langle \bar{v},x_i\rangle>0$, we instead get
\begin{align*}
    f_{\bar{u},\bar{v}}(x_i) = \sigma(\sigma(\langle \bar{u},x_i\rangle) - \sigma (\langle \bar{v},x_i\rangle)) = \sigma(\langle \bar{u},x_i\rangle- \langle \bar{v},x_i\rangle) = c\langle \bar{u}-\bar{v},x_i\rangle.
\end{align*}
Again due to linear independence of the $x_i$, the values determine the values of $\bar{u}-\bar{v}$, and from that $\bar{v}=-(\bar{u}-\bar{v})+\bar{u}$.

Now let $\tau$ be any linear functional on $\R^2\times \R^2$, and define 
\begin{align*}
    \alpha := \inf_{\bar{u},\bar{v}\in \widetilde{\calN}} \tau(\bar{u},\bar{v})
    \quad \quad \quad
    \beta := \sup_{\bar{u},\bar{v},\in \widetilde{\calN} }\tau(\bar{u},\bar{v}).
\end{align*}
Since $\widetilde{\calN}$ is not contained in a lower-dimensional subspace, $\alpha<\beta$. Define
\begin{align*}
    \widetilde{\mathcal{A}} & := \tau^{-1}\left(\bigcup_{m\geq 1} \left( \alpha + (\beta-\alpha)\cdot \left(\frac{1}{2m},\frac{1}{2m+1}\right) \right ) \right), \\
    \ \widetilde{\mathcal{B}} &:= \tau^{-1}\left(\bigcup_{m\geq 0} \left( \alpha + (\beta-\alpha)\cdot \left(\frac{1}{2m+1},\frac{1}{2m}\right)  \right ) \right) .
\end{align*}
These subsets of $\widetilde{\mathcal{N}}$ are open and disjoint, with common boundary
\begin{align*}
    \widetilde{\mathcal{K}}:=\tau^{-1}\left(\left\{\alpha + (\beta-\alpha)\cdot \frac{1}{n} \mid n \in \N \right\}\right).
\end{align*}

We claim that $\widetilde{\calA}$ is not semi-algebraic. In order to prove this, consider a line $\lambda \sse \R^2 \times \R^2$ so that
\begin{enumerate}
\item $\tau(\lambda \cap \widetilde{\calN})= (\alpha,\beta)$;\item $\tau$ restricted to $\lambda$ is injective. 
\end{enumerate}
Assume by contradiction that $\widetilde{\calA}$ is semi-algebraic. Then $\widetilde {\calA} \cap \widetilde{\calN}$ would also be. As $\lambda$ is a line, it is semi-algebraic and since $\tau$ is a semi-algebraic bijection on $\widetilde{\calN} \cap \lambda$ the set 
\begin{align*}
    \tau (\widetilde{\calA} \cap \lambda) = \left(\bigcup_{m\geq 1} \left ( \alpha + (\beta-\alpha)\cdot \left(\frac{1}{2m},\frac{1}{2m+1}\right) \right ) \right)
\end{align*}
would also be semi-algebraic. But this is a contradiction, as the above set is clearly not semi-algebraic, since it contains infinitely many connected components.

Define $\widetilde{\calV}_0:=\widetilde{\calA} \cup \widetilde{T}(\widetilde {\mathcal{B}})$. It is easy to show that $\phi$ is injective on $\widetilde{\calV}_0$. From this point on, we move back to the entire parameter space  via $\psi$. To this end, we set: 
\begin{align*}
    \mathcal N & :=\psi^{-1}(\widetilde{\mathcal N}) \quad  \quad
     \mathcal A:=\psi^{-1}(\widetilde{\mathcal A})  \quad  \quad \mathcal K:=\psi^{-1}(\widetilde{\mathcal K})
    \\
    \mathcal B & :=\psi^{-1}(\widetilde{\mathcal B})  \quad  \quad 
    \mathcal V_0:=  \psi^{-1}(\widetilde{\mathcal V}_0) 
\end{align*}
Define $\mathcal V:= \Sigma_2 \cdot \mathcal V _0$. It is open, and furthermore  $G(\mathbf{d})$-invariant: indeed it is $\Sigma_2$-invariant by construction, and $\mathrm{Sc}(\mathbf{d})$-invariant due to the structure of the fibers of $\psi$. Furthermore, it follows from the fact that the parametrization map $\phi$ is injective on $\widetilde{\mathcal V} _0$ and $\Sigma_2$-invariant that $ \mathcal V $ is weakly identifiable. 

We claim that given any open weakly identifiable subset $\calV \subseteq \calW \subseteq \R^4\times \R^2 \times \R$, we have $\mathcal W \cap \mathcal N = \mathcal A$.  First, $\mathcal A \subseteq \mathcal V_0 \subseteq \mathcal V \subseteq \mathcal W$, so $\mathcal A \subseteq \mathcal W \cap \mathcal N$. Next let us prove that $\calW \cap \mathcal B =\emptyset$. Assume by contradiction there exists an element $b \in \mathcal B \cap \calW$, then $ \widetilde{b}:= \psi(b)\in \widetilde{\mathcal B}$. Choose $b' \in \psi^{-1}(\widetilde{T}(\widetilde{b}))$. As $\widetilde{T}(\widetilde{b}) \notin \Sigma_2 \cdot \widetilde{b}$, it follows that $b'  \notin G(\mathbf{d}) \cdot b$, but since $\widetilde{T}(\widetilde{b})\in \widetilde{T}(\widetilde{\mathcal{B}})\subseteq \widetilde{\mathcal V _0}$, we have $b' \in \mathcal{V}_0 \subseteq \calW$. Now, notice that $b$ and $b'$ realize the same function, in fact 
\[
    \Phi(b) = \phi(\widetilde{b})= \phi(\widetilde{T}(\widetilde{b}))= \Phi(b')
\]
which is a contradiction to the weak identifiably of $\calW$. Hence, it must be $\mathcal{W} \cap \mathcal{B} = \emptyset$

Finally let $k \in \mathcal K$ and assume $k \in \mathcal W$, then $\mathcal W \cap \mathcal N$ is an open neighborhood of $k$ in $\mathcal N$. As $k$ is in the closure of $\mathcal B$, this would imply that $\mathcal W \cap \mathcal B \neq \emptyset$, which contradicts what we have just proven above. Thus $\mathcal W \cap \mathcal K = \emptyset$ and we have proven that $\mathcal W \cap \mathcal N = \mathcal A$.

Consider now the partially ordered family 
\[
    \mathfrak W := \{\mathcal W \subseteq \R^4\times \R^2 \times \R \mid \mathcal V \subseteq \mathcal W, \text{ $\mathcal W$ is open and weakly identifiable}\}.
\]
The family is not empty as $\mathcal V \in \mathfrak W$. Given any chain $\{\mathcal W_i\}_{i \in I}$ in $\mathfrak W$ it is straightforward to check that its union $\mathcal W _\infty$ is open and weakly identifiable. As every chain in $\mathfrak W$ has an upper bound, by Zorn's lemma $\mathfrak W$ has a maximal element that we denote by $\mathcal U$. From the fact that any weakly identifiable open subset of $\R^4\times \R^2 \times \R$ containing properly $\mathcal U$ would also contain $\mathcal V$, it follows that $\mathcal U$ is maximal among all weakly identifiable open subsets of $\R^4\times \R^2 \times \R$.

We conclude the counterexample by showing that $\mathcal U$ is not semi-algebraic. Since $\mathcal V \subseteq \mathcal U$, we have that $\mathcal U \cap \mathcal N = \mathcal A$. Since $\mathcal N$ is semi-algebraic, if $\mathcal U$ was semi-algebraic then their intersection would be as well, but $\mathcal U \cap \mathcal N = \mathcal A$ and $\mathcal A$ is not semi-algebraic as otherwise we would have that $\widetilde{\mathcal A}$ is semi-algebraic. Thus $\mathcal U$ is not semi-algebraic. 

\section{Shallow ReLU networks}\label{sec:shallow-networks}

In this section we focus on the study of the neuromanifold induced by the simplest architecture possible, the case when $L=1$, i.e.\ when there is only one hidden layer. These are called \emph{shallow} neural networks. Here, we describe explicitly $\cI_{\varphi}$ in the shallow case and we conclude giving an example where Theorem \ref{theo:schie} does not hold.

As above, we will focus on the case of zero bias. Since we are in the shallow case, the parametrization morphism 
\[
    \Phi: \R^M \longrightarrow C^0(\R^{d_0})
\]
can be described explicitly as follows. As in the shallow case $M=d_0d_1+d_1$, a vector in $\R^M$ can be thought as a pair $(A,b)$ where $A \in \textup{Mat}_{d_1,d_0}(\R)\simeq (\R^{d_0})^{d_1}$ and $b\in \R^{d_1}$. The function $\Phi(A,b)$ is defined by the rule
\[
    x \mapsto \sum_{i=1}^{d_1} b_i \sigma( a_i \cdot x)
\]
where $a_i$ is the $i$-th row of the matrix $A$ and $\sigma$ is the ReLU activation function. 

Using the notation and reduction steps described in \Cref{sec:identifiable-symmetries}, it is enough to study the (image of the) morphism 
\[
    \varphi: (S \times \R)^{d_1} \longrightarrow C^0(\R^{d_0})
\]
where $S:=S^{d_0-1}\simeq (\R^{d_0}\setminus 0)/(\R_+)$ is the $(d_0-1)$-dimensional sphere. The morphism $\varphi$ is $\Sigma_{d_1}$-invariant, where the action is given by permutation of the $d_1$-fold product of $S\times \R$. Notice that since $\Sigma_{d_1}$ is a finite group, if the neural network did not have other symmetries, we would end up with a semi-algebraic quotient (thanks to \Cref{lem:topology} and \Cref{theo:schie}). This is not the case, as we are going to show.

We can describe $\varphi$ as follows: given an element $(v,b) \in (S \times \R)^{d_1}$, we have that the function $\varphi( (v,b))$ is defined as
\[
    x \mapsto \sum_{i=1}^{d_1} b_i \sigma( v_i \cdot x), 
\]
analogously to $\Phi$,
where all we have done is just restricting the parametrization morphism to the case of unit vectors, since $\sigma$ is positively homogeneous.

\begin{definition}
    Given a piecewise linear function $f$, we denote by $\Gamma_f$ the \emph{skeleton} of $f$, i.e.\ the (possibly empty) closed subset of $\R^d$ where the function $f$ is not differentiable. 
\end{definition}
The next lemma, although very simple, perfectly depicts the strong non-linearity of piecewise linear functions, which we believe is the reason why the neuromanifold cannot be endowed with a global affine semi-algebraic structure. 
\begin{lemma}\label{lem:high-non-linearity}
    Let $d$ be a positive integer, $b:=( b_1, \dots , b_{d})  \in \R^{d}$ and $f_i \in \PL(\R^n)$ be piecewise linear functions with skeleton $\Gamma_i$ for $1\leq i\leq d$. Suppose that $\Gamma_i \not\subset \cup_{j\neq i} \Gamma_j$ for every $1\leq i\leq d$ and that 
    \[
        \sum_{i=1}^{d} b_i f_{i} = 0
    \]
    as a function in $C^0(\R^n)$. Then $b_i=0$ for $i=1,\dots, d$.
\end{lemma}

\begin{proof}
    Fix $i$ and let $x_i\in \R^n$ be a point in $\Gamma_i \setminus (\cup_{j\neq i} \Gamma_j)$. Let us define the operator 
    $\dna_{w}(-)$ on a function $f$ for a direction $w \in S^{n-1}$ as 
    $$
    \dna_wf(x) := \lim_{t\to 0^+} \left ( \partial_{-w}f(x)-\partial_{-w}f(x+tw) \right )
    $$
    Notice that the operator is clearly linear, thus if we apply it to the linear combination, we get
    \[
        \sum_{i=1}^{d} b_i  (\dna_{w}{f_{i}}) = 0.
    \]
    Since $f_j$ is linear around $x_i$ for every $j \neq i$, we get that $\dna_{w}f_j(x_i)= 0$ for any $w \in S^{n-1}$.  Moreover, since $x_i$ is in the skeleton $\
    \Gamma_i$, there exists a $w$ such that $\dna_{w}{f_{i}}(x_i)\neq 0$. Thus $b_i=0$ and we can conclude.
\end{proof}

 \Cref{lem:high-non-linearity} allows us to deduce the linear independence of piecewise linear functions when the skeleta are sufficiently different. In the shallow case, it gives us that any collection of pairwise linearly independent unit vectors will give rise to a system of linearly independent functions. From now on, let us denote by $f_{v}$ the function $ \sigma( \langle v , x \rangle)$ for a unit vector $v \in S$.

It is crucial to understand what happens if the linear independence condition on the unit vectors is not met, as the following remark explains.
 \begin{remark}
       Consider the following sign-splitting intrisic symmetry  of the ReLU function
 $$ \sigma(x) - \sigma(-x) = x, \quad \forall \, x \in \R.$$ 
 It immediately implies that 
 \begin{align}
     f_v(x) - f_{-v}(x)=\sigma(\langle v,x\rangle) - \sigma(-\langle v,x \rangle) = \langle v,x \rangle  \quad \forall \, x \label{eq:relu_symmetry}
 \end{align} 
and thus the strong linear independence fails, since we enter the realm of linear functions.

Equation \eqref{eq:relu_symmetry} above has been identified before, see \cite{vlavcic2021affine,agrawal2022classification}. In our language, it produces hidden symmetries in the network. Indeed, let $v_1, \dots, v_{d_1}$ be linearly dependent vectors satisfying
\begin{align}
    \sum_{i=1}^{d_1} b_i v_i = 0 
\end{align}
for some $b_1, \dots ,b_{d_1} \in \R$.
Then $\sum_{i=1}^{d_1} b_i (f_{v_i}(x)-f_{-v_i}(x)) = \sprod{\sum_{i=1}^{d_1} b_i v_i,x}=0$, 
so that 
\begin{align} \label{eq:hidden_symmetry}
    \sum_{i=1}^{d_1} b_i f_{v_i}(x) = \sum_{i=1}^{d_1} b_i f_{-v_i}(x).
\end{align}
Note that for generic $b_i$ and $v_i$, the parameter $(v_i,b_i)$ and $(-v_i,b_i)$ are not related through an element in $G(\mathbf{d})$, so that the symmetry indeed is hidden.
\end{remark}

In what follows, we define a semi-algebraic subset $H_0$ of $\cP$ and prove that it is equal to $\cI_{\varphi}$. We construct it as the complement of the following $\Sigma_{d_1}$-equivariant semi-algebraic subsets of $\cP$, accounting for hidden symmetries:
\begin{align*}
    \Delta & := \{(v_1, \dots , v_{d_1}) \in S^{d_1}\mid \exists \,i\neq j \in [d_1]\textup{ such that } v_i = v_j\} \times \R^{d_1}; \\
    \Omega & := S^{d_1} \times \{(b_1, \dots , b_{d_1}) \in \R^{d_1}\mid \exists \,i \in [d_1] \textup{ such that } b_i=0\} ; \\
    \Lambda & := \{ (v,b) \in \cP \mid \exists \,I \subseteq [d_1], \textup{ such that } \textstyle\sum_{i\in I} b_i v_i=0 \text{ and  $(b_i,v_i)\neq (b_j,-v_j)$ for $i,j\in I$} \}; \\
    \Xi & := \{(v,b)\in \cP \mid  \exists \,i \neq j \in [d_1]\textup{ such that } v_i = -v_j,\, b_i = -b_j \},
\end{align*}
and an additional set $\Upsilon$, that is more technical to define. Given $(v,b) \in \cP$, let us call two indices $j,h \in [d_1]$ \emph{antipodal} if
$v_j=-v_h$. Without loss of generality we can consider antipodal indices to come in pairs. If they do not, $(v,b)\in \Delta$ and we treat this case differently in the following. Let $K$ be the set of antipodal pairs $\{(j,h) \mid v_j=-v_h\}$: it can be empty, in fact generically it is. We choose a representative in each antipodal pair of $K$ and denote as $K_+:=\{i_1^+,    \dots, i_s^+\}$ the set of such indices and as $K_-:=\{i_1^{-},    \dots, i_s^-\}$ the remaining indices.  Now, $\Upsilon$ is the set of points $(v,b)\in \cP$ for which there exists a partition of antipodal pairs $K = K' \sqcup K''$, $K'\neq \emptyset$,  with the following property: there exists a choice of representatives for $K$ and subsets $J',J'' \subseteq [d_1] \setminus (K_+ \cup K_-)$ of non-antipodal indices with $J' \cap J''=\emptyset$ and $|J''| \le |K'|$ such that 
\[
    \sum_{i^+\in K_+'}b_{i^+}v_{i^+}  + \sum_{j \in J'} b_jv_j \in \mathrm{span} (\{v_{i} \mid i \in K''_+  \cup J''\}),
\]
where $K_+'$ and $K_+''$ are the set of representatives in $K'$ and $K''$, respectively.

\begin{remark}
    Observe that $\Delta$, $\Omega$, $\Lambda$ and $\Xi$ are Zariski closed sets. We will see in the following that $\Upsilon$ in general is only a semi-algebraic set.
\end{remark}

\begin{definition}
    Let $d_1$ be a positive integer. If $d_1 \geq 3$, we set $H_0$ to be the intersection of the complements of the five closed subsets defined above. If $d_1=1,2$, we set $H_0$ to be the intersection of the complements of $\Delta$, $\Omega$ and $\Lambda$.
\end{definition}

\begin{proposition}\label{prop:shallow-max-identifiable}
  The inclusion $\cI_{\varphi} \subseteq H_0$ holds. 
\end{proposition}

\begin{proof}
      We focus on the case $d_1 \ge 3$, the proof of the cases $d_1 =1,2$ is straightforward.
      
      It is enough to show that for every element $w$ belonging to the union of the five sets above the action of $\Sigma_{d_1}$ on $\varphi^{-1}\varphi(w)$ is not transitive. That is, we need to show the existence of hidden symmetries. We treat the five cases separately.

     \underline{Case 1: $w \in \Delta$.} Clearly, it is enough to prove that $\varphi^{-1}\varphi(w)$ is not finite. By construction, $w:=(v,b) \in S^{d_1}\times \R^{d_1}$ such that there exists two distinct indexes $i\neq j \in [d_1]$ such that $v_i=v_j$. Consider now the elements $(v, b')$ such that $b'_h:=b_h$ for $h\neq i,j$ and $b'_i+b'_j=b_i+b_j$. We then have
     \begin{align*}
         \varphi(v,b) = \sum_{h \neq i,j }b_h f_{v_h} + (b_i+b_j)f_{v_j} =   \sum_{h \neq i,j }b_h' f_{v_h'} + (b_i'+b_j')f_{v_j} = \varphi(v,b'),
     \end{align*}
     i.e. $(v,b)$, $(v',b)\in \varphi^{-1}\varphi(x)$. Since $b_i',b_j'$ vary in a 1-dimensional subspace of parameters, the fiber is infinite.

     \underline{Case 2: $w \in \Omega$.} As above, we prove that $\varphi^{-1}\varphi(x)$ is not finite. By construction, $w:=(v,b) \in S^{d_1}\times \R^{d_1}$ such that there exists $i\in [d_1]$ such that $b_i=0$. Consider now the elements $(v', b)$ such that $v_h':=v_h$ for $h\neq i$ and $v_i' \in S$. We have
     \begin{align*}
         \varphi(v,b) = \sum_{h\neq i} b_h f_{v_h}
         = \varphi(v',b) 
     \end{align*}
     for all such elements. As $v_i'$ varies in $S$, the fiber is infinite.

    \underline{Case 3: $w \in \Lambda$.} By definition of $\Lambda$, 
    there exists $I \subseteq [d_1]$ 
    such that $\sum_{i \in I} b_iv_i=0$. We define $(v',b') \in \cP$ setting
     \begin{itemize}
         \item $v_i':=-v_i$ and $b_i':=b_i$ for every $i\in I$;
         \item $v_i':=v_i$ and $b_i':=b_i$ for every $i \notin I$.
     \end{itemize} 
     We then get
     \begin{align*}
         \varphi(v',b') = \sum_{i\notin I}b_i'f_{v_i}  + \sum_{i \in I}b_i f_{-v_i} = \sum_{i\notin I}b_i'f_{v_i}  + \sum_{i \in I}b_i f_{v_i} = \varphi(v,b),
     \end{align*}
     where we used equation \eqref{eq:hidden_symmetry}. The parameter $(v',b')$ can only be obtained from $(v,b)$ by a permutation if there are indices $i\neq j$ with $b_i=b_j$ and $v_i=-v_j$. But this is excluded by definition of $\Lambda$.
     
     \underline{Case 4: $w \in \Xi$.} By definition of $\Xi$, $w:=(v,b) \in S^{d_1}\times \R^{d_1}$ such that there exists two distinct indexes $i\neq j \in [d_1]$ such that $v_i=-v_j$ and $b_i=-b_j$. Since $d_1\geq 3$, there exists a third index $h\neq i,j$. Consider now the vector $u:=b_iv_i + b_hv_h$. It cannot be parallel to $v_i$ as in that case we either have $b_h=0$, which can be handled as in Case $2$, or $v_h$ must be either $v_i$ or $-v_i=v_j$. In both those latter cases, we get a repeated $v_i$, and we can proceed as in Case 1. 

     In particular, $u$ cannot be the zero vector, so that $u/\norm{u}$ is well defined. Now notice that
     \begin{align*}
         -\norm{u}\cdot \frac{u}{\norm{u}} + b_i v_i + b_h v_h=0, 
     \end{align*}
     so equation \eqref{eq:hidden_symmetry} implies 
     \begin{align} \label{eq:cashin_symmetry}
          \norm{u}f_{-u/\norm{u}} + b_i f_{v_i} +  b_h f_{v_h} =   \norm{u}f_{u/\norm{u}} + b_i f_{-v_i} +  b_h f_{-v_h}.
     \end{align}
     Define $(v',b') \in \cP$ by setting
     \begin{itemize}
         \item $v_h':=-v_{h}$, $b_h':=b_h$
         \item $v_i':=u/\norm{u}$, $b_i':=\norm{u}$;
         \item $v_j':=-u/\norm{u}$, $b_j':=-\norm{u}$;
         \item $v_{g}'=v_g$ and $b_g'=b_g$ for every $g\notin\{i,j,h\}$.
     \end{itemize}
     We then have
     \begin{align*}
         \varphi(v',b') &= \norm{u}f_{u/\norm{u}} - \norm{u}f_{-u/\norm{u}} + b_h f_{-v_{h}} +\sum_{g\notin\{i,j,h\}} b_gf_{v_g} 
         \\ & \stackrel{\eqref{eq:cashin_symmetry}}{=} b_if_{v_i}-b_if_{-v_i}+ b_h f_{v_{h}} +\sum_{g\notin\{i,j,h\}} b_gf_{v_g} = \varphi(v,b) .
     \end{align*}
     It is now only left to prove that $(v',b')$ is not related to $(v,b)$ via a permutation. Suppose that it was. Then, $v_i$ must be equal to either $u/\norm{u}$, $-v_h$ or one of the $v_g$, $g\notin \{i,j,h\}$. We already argued the first case cannot happen, while the third case would give a repeated vector that can be handled as in Case 1. In the second case, we would get $v_j=-v_i=v_h$, which can be handled analogously.

    \begin{figure}[h]
    \centering
    \begin{tikzpicture}[
    >={Latex},
    thick,
    font=\small
]


\colorlet{Kcolor}{blue!20}
\colorlet{Kcolor2}{blue!25}
\colorlet{MovedK}{red!20}
\colorlet{PhiColor}{red!35}
\colorlet{Jcolor}{orange!25}
\colorlet{Jcolor1}{orange!35}
\colorlet{Jcolor2}{purple!15}


\draw (-6,1.8) rectangle (0,3.0);

\draw (-3,1.8)--(-3,3.0);
\draw (-6,2.4)--(0,2.4);

\fill[MovedK] (-6,2.4) rectangle (-3,3.0);
\fill[Kcolor] (-3,2.4) rectangle (0,3.0);
\fill[Kcolor2] (-6,1.8) rectangle (-3,2.4);
\fill[Kcolor2] (-3,1.8) rectangle (0,2.4);

\node at (-4.5,2.7) {$K'_+$};
\node at (-1.5,2.7) {$K'_-$};
\node at (-4.5,2.1) {$K''_+$};
\node at (-1.5,2.1) {$K''_-$};

\draw (0.8,2.4) rectangle (6.8,3.0);

\fill[Jcolor] (0.8,2.4) rectangle (2.8,3.0);
\fill[Jcolor1] (2.8,2.4) rectangle (4.8,3.0);
\fill[Jcolor2] (4.8,2.4) rectangle (6.8,3.0);

\draw (2.8,2.4)--(2.8,3.0);
\draw (4.8,2.4)--(4.8,3.0);

\node at (1.8,2.7) {$J'$};
\node at (3.8,2.7) {$J''$};
\node at (5.8,2.7) {$J'''$};

\node at (-7,2.4) {$(b,v)$};


\draw (-6,-0.8) rectangle (0,0.4);

\draw (-3,-0.8)--(-3,0.4);
\draw (-6,-0.2)--(0,-0.2);

\fill[gray!10] (-6,-0.2) rectangle (-3,0.4);

\fill[Kcolor] (-3,-0.2) rectangle (0,0.4);
\fill[Kcolor2] (-6,-0.8) rectangle (-3,-0.2);
\fill[Kcolor2] (-3,-0.8) rectangle (0,-0.2);

\node at (-1.5,0.1) {$K'_-$};
\node at (-4.5,-0.5) {$K''_+$};
\node at (-1.5,-0.5) {$K''_-$};

\draw (0.8,-0.2) rectangle (6.8,0.4);

\fill[Jcolor] (0.8,-0.2) rectangle (2.8,0.4);
\fill[Jcolor1] (2.8,-0.2) rectangle (4.8,0.4);
\fill[Jcolor2] (4.8,-0.2) rectangle (6.8,0.4);

\draw (2.8,-0.2)--(2.8,0.4);
\draw (4.8,-0.2)--(4.8,0.4);

\node at (1.8,0.1) {$J'$};
\node at (3.8,0.1) {$J''$};
\node at (5.8,0.1) {$J'''$};

\node at (-7,-0.2) {$(b',v')$};


\fill[MovedK] (2.8,-.8) rectangle (6.6,-0.2);
\draw         (2.8,-.8) rectangle (6.6,-0.2);

\fill[PhiColor] (2.8,-.8) rectangle (4.8,-0.2);

\draw[dashed] (4.8,-0.8) -- (4.8,-0.2);

\node at (3.8,-0.5)
{$\varpi(J'')$};

\node at (5.7,-.5)
{\scriptsize $K'_+\setminus\varpi(J'')$};

\end{tikzpicture}
    \caption{The construction used in Case 5. The indices in $K_+'$ are moved to form new antipodal pairs matching the indices in $J''$. This only works when $|J''|\leq |K'$. Best viewed in color.}
    \label{fig:placeholder}
\end{figure}

     \underline{Case 5: $w\in \Upsilon$} By definition, there exists a partition $K'\sqcup K''$ of $K$, two disjoint sets $J'$ and $J''$ of non-antipodal indices with $|J''|\le |K'|$, and a choice of representatives such that
     \begin{align}
         \sum_{i^+\in K_+' }b_{i^+}v_{i^+} +\sum_{j\in J' }b_{j}v_{j} + \sum_{i^+ \in K_{+}''}\lambda_{i^+} v_{i^+} + \sum_{j \in J''} \mu_j v_j = 0 \label{eq:lincomb}
     \end{align}
     for real numbers $\lambda_{i+}$ and $\mu_j$. Without loss of generality, we can assume that all of these coefficients are non-zero: any index with a zero coefficient can be removed from $K_{+}''$ or $J''$, respectively. For any antipodal pair $(i^+,i^-) \in K$ the following holds
\begin{align}\label{eq:change1}
    \left(b_{i^{+}} f_{v_{i^{+}}}(x) +b_{i^{-}}f_{v_{i^{-}}}(x)\right) -(b_{i^{+}}+b_{i^{-}})f_{v_{i^{-}}}(x) &= b_{i^+}f_{v_{i^{+}}} (x)-b_{i^+}f_{-v_{i^{+}}}(x) = \langle b_{i^+} v_{i^{+}},x\rangle.
\end{align}
    In a similar way, we obtain 
\begin{align}\label{eq:change2}
         \left(b_{i^+} f_{v_{i^+}}(x) +b_{{i^-}}f_{v_{i^-}}(x)\right)  - \left((b_{i^+}-\lambda_{i^+})f_{v_{i^{+}}}(x) + (b_{i^-}+\lambda_{i^+})f_{v_{i^-}} \right)
        &= \langle \lambda_{i^+} v_{i^+},x\rangle  \\
\label{eq:change3}
         b_j f_{v_j}(x) - \left((b_j-\mu_j)f_{v_j}(x) +\mu_j f_{-v_j}(x)\right) &= 
          \sprod{\mu_jv_j,x},
\end{align}
as well as the already noted
\begin{align} \label{eq:change4}
      b_j f_{v_j}(x) - b_jf_{-v_j}(x) = \sprod{b_jv_j,x},
\end{align}
     Now, since $|J''|\leq |K'|$, there exists an injective map $\varpi \co J'' \hookrightarrow  K'$. We use this map to define a new parameter $(v',b')\in \cP$. Notice that since $\varpi$ is injective, the assignment below is well-defined:
    \begin{itemize}
        \item $v_{i^+}':=v_{i^+}$, $b_{i^+}':=b_{i^+}- \lambda_{i^+}$, $v_{i^-}':=v_{i^-}$, $b_{i^-}':=b_{i^-}+ \lambda_{i^+}$ if $(i^+,i^-)\in K''$;
        \item $v_{i^-}':=v_{i^-}$, $b_{i^-}^' := (b_{i^+}+b_{i^-})$, if $i^- \in K'_-$;
        \item $v_{i^+}':=v_{i^+}$, $b_{i^+}':=0$ if  $i^+ \in K'_+ \setminus \varpi(J'')_+$ ;
        \item $v_{i^+}' :=-v_j$, $b_{i^+}':=\mu_j$ if  $i^+ = \varpi(j)^+ \in \varpi(J'')_+$;
        \item $v_j':= -v_j$, $b_j':=b_j$ for $j\in J'$;
        \item $v_j' := v_{j}$, $b_j' := b_j-\mu_j$ for $j\in J''$;
        \item $v_j':=b_j$, $b_j':=b_j$ for $j\in J''':=[d_1]\setminus(K_+\cup K_{-}\cup J'\cup J'')$.
    \end{itemize}
Using equations \eqref{eq:change1}, \eqref{eq:change2}, \eqref{eq:change3} and \eqref{eq:change4} we obtain:
    \begin{align*}
        \sum_{i \in [1,\dots, d_1]} b_{i}'f_{v_{i}'}(x) =&  \sum_{(i^+,i^{-})\in K'} (b_{i^+}+ b_{i^-})f_{v_{i^-}}(x)
        \\
        + & \sum_{(i^+,i^{-})\in K''} (b_{i^+}- \lambda_{i^+})f_{v_{i^+}}(x) + (b_{i^-} + \lambda_{i^+})f_{v_{i^-}}(x) 
        \\
        + & \sum_{j\in J'} b_j f_{-v_j}(x) + \sum_{j \in J''} ( b_j-\mu_j) f_{v_j}(x) + \mu_j f_{-v_j}(x)  + \sum_{j \in J'''} b_j f_{v_j}(x)
        \\
        =&  \sum_{(i^+,i^{-})\in K'} b_{i^+}f_{v_{i^+}}(x)+ b_{i^-}f_{v_{i^-}}(x) - \sprod{b_{i^+}v_{i^+},x}  
        \\
        + & \sum_{(i^+,i^{-})\in K''} b_{i^+}f_{v_{i^+}}(x) + b_{i^-} f_{v_{i^-}}(x) - \sprod{\lambda_{i^+}v_{i^+},x}  \\
        + & \sum_{j\in J'} b_j f_{v_j}(x) -\sprod{\mu_jv_j,x} + \sum_{j\in J''}b_j f_{v_j}(x) -\sprod{b_jv_j,x} + \sum_{j \in J'''} b_jf_{v_j}(x)
        \\
        =&  \sum_{i \in [1,\dots, d_1]} b_{i}f_{v_{i}}(x) - \left\langle \sum_{i_+\in K_+' }b_{i^+}v_{i^+} + \sum_{j\in J'}b_j v_j + \sum_{i_+ \in K_{+}''}\lambda_{i^+} v_{i^+} + \sum_{j \in J''}\mu_j v_j, x\right\rangle  
        \\
        = & \sum_{i \in [1,\dots, d_1]} b_{i}f_{v_{i}}(x) .
    \end{align*}
    Thus $(v',b')$ lies in the fiber of $(v,b)$ and the two are not in the same $\Pr(\mathbf d)$-orbit.

\end{proof}
\begin{proposition}\label{prop:strongly-identifiable}
     With the notation above, the $\Sigma_{d_1}$-equivariant semi-algebraic subset $H_0 \subseteq  \cP$ is strongly identifiable, i.e.\ $H_0 \subseteq s\calI_\varphi$. 
\end{proposition}

\begin{proof}
As for the previous proposition, we focus on the case $d_1 \ge 3$. 

Since $H_0$ is in the complement of $\Delta \times \R^{d_1}$, the action of $\Sigma_{d_1}$ is free, and it is enough to prove that $H_0$ is identifiable. Let $(v,b)\in U_0$ and $(w,c) \in \cP$ such that $\varphi(v,b)=\varphi(w,c)$. We then have
    \begin{align} \label{eq:difference_function}
        g:=\sum_{(i^+,i^-)\in K }(b_{i^+}f_{v_{i^+}}+b_{i^-} f_{v_{i^{-}}}) +   \sum_{j \in J} b_jf_{v_j} - \sum_{k \in [d_1]} c_k f_{w_k} = 0
    \end{align}
    for $K$, $K_+$ and $K_-$ as in the definition of $\Upsilon$ and $J:=[d_1]\setminus (K_+\cup K_{-})$.
    Let us call an index $k$ \emph{attached to the index} $j \in J$ if $w_k = \pm v_j$, and \emph{attached to the pair} $(i^+,i^{-}) \in K$ if $w_k \in \{v_{i^+}, v_{i^{-}}\}$. For an index $j \in J$ or a pair $(i^+,i^{-}) \in K$ let $A_{j}$, respectively $A_{(i^+,i^{-})}$, be the set of indices $k$ attached to that index or pair.

    We prove that
       \begin{itemize}
            \item[(a)] $|A_{(i^+,i^{-})}|=2$ for all $(i^+,i^{-}) \in K$;
            \item[(b)] $|A_{j}|=1$ for all $j \in J$;
            \item[(c)] for every $k \in [d_1]$ there exists either a pair $(i^+,i^{-}) \in K$ or an index $j \in J$ such that $k \in A_{(i^+,i^{-})}$ or $k \in A_j$.
       \end{itemize}
        
       First, let us notice that since $(v,b)\notin \Delta$, condition (c) must hold. Furthermore, for each $(i^+,i^{-}) \in K$ and $j \in J$ the corresponding set $A$ has at least cardinality 1. If this was not the case, since $(v,b)\notin \Omega\cup \Xi$, the difference function will have a non-trivial skeleton component in the set $\{\langle v_j,x\rangle=0\}$ or $\{\langle{v_{i^+},x}\rangle=0\}$, which would contradict equation (\ref{eq:difference_function}). By the same argument, for any $k \notin \cup_{K \cup J} A$, we must have $c_k=0$.
       
       Let us now study the functions
       \begin{align*}
           g_{(i^+,i^{-})} &:= b_{i^+}f_{v_{i^+}}+b_{i^-} f_{v_{i^{-}}} - \sum_{k \in A_{(i^+,i^{-})}}c_k f_{w_k} \\
            g_{j} &:=  b_f f_{v_j} - \sum_{k \in A_j}c_k f_{w_k}
       \end{align*}
       for $(i^+,i^{-}) \in K$ and $j \in J$. By the the same argument as above, each of these functions is forced to be linear. 
       Notice that
       \begin{enumerate}
           \item[(i)] for all $(i^+,i^{-}) \in K$ and $j \in J$, there exist $\lambda_{i^+}, \mu_j \in \R$ such that  $g_{(i^+,i^{-})} = \sprod{ \lambda_{i^+} v_{i^+},\cdot}$  and $g_{j}= \sprod{ \mu_j v_j,\cdot}$;
           \item[(ii)] if $|A_{(i^+,i^{-})}|=1$ for $(i^+,i^{-}) \in K$, then $g_{(i^+,i^{-})} \in \{\sprod{b_{i^+}v_{i^+},\cdot}, \sprod{b_{i^{-}}v_{i^{-}},\cdot} \}$  and if $|A_j|=1$ for $j \in J$, then $g_j = \sprod{b_jv_j,\cdot}$. 
       \end{enumerate} 

       Define the following partition of $K$: \begin{align*}
       K'&:=\{(i^+,i^{-}) \in K \mid |A_{(i^+,i^{-})}|=1 \}
       \\
       K''&:=\{(i^+,i^{-}) \in K \mid |A_{(i^+,i^{-})}|\ge 2 \}.
       \end{align*} 
       Analogously, partition $J$ as:
       \begin{align*}
           J'&:=\{j \in J \mid |A_{j}|=1 \}
       \\
       J''&:=\{j \in J \mid |A_{j}|\ge 2 \}.
       \end{align*} 
       Suppose that $K' \neq \emptyset$. It follows from (ii) that
       \begin{align}\label{eq:contradict-upsilon}
           0 = g(x) = \left\langle\sum_{i^+\in K'_+} b_{i^{+}}v_{i^+}  + \sum_{j \in J'} b_jv_j,x\right\rangle + \sum_{i^+ \in K''_+}g_{(i^+,i^{-})}(x) + \sum_{j \in J''}g_j(x)   
       \end{align}
       where we have picked the other representative for any pair in $K'$ such that $g_{(i^+,i^{-})} = \sprod{b_{i^{-}}v_{i^{-}},\cdot}$, and used that $c_k=0$ for indices $k$ not belonging to any set $A$. Equation (\ref{eq:contradict-upsilon}) above  is a contradiction to $(v,b)\notin \Upsilon$ as long as $|J''|\leq |K'|$. Assume by contradiction that $|J''|>|K'|$. Since there exists one index $k$ for every pair in $K'$ and index in $J'$, at least two for every pair in $K''$ and index in $J''$, and since $[d_1] = K'_+\sqcup K'_{-} \sqcup K''_{+}\cup K''_{-}\sqcup J'\sqcup J''$ we get 
       \begin{align*}
           d_1 \geq  |K'|+2|K''| + |J'|+2|J''|>2|K'|+2|K''| + |J'|+|J''|  = d_1.
       \end{align*}
       which is a clear contradiction.
       Hence, $K' =\emptyset$, and $|A_{(i^+,i^{-})}|\ge 2$ for each pair $(i^+,i^{-}) \in K$. 
       Suppose there exists $(i^+,i^{-}) \in K$ such that $|A_{(i^+,i^{-})}|> 2$. With a similar argument as above, we obtain
       \begin{align*}
           d_1 >  2|K''| + |J'|+2|J''| = d_1,
       \end{align*}
       which again is a contradiction. Thus $K''=\{(i^+,i^{-}) \in K \mid |A_{(i^+,i^{-})}|= 2 \}$. Once more, with the exact same argument, we prove that $J'' = \emptyset$. Hence, claims (a), (b), (c) hold.

       Lastly, we claim that all the functions $g_{(i^+,i^{-})}$ and $g_j$ are zero. If this was not the case, let $K^\circ \subseteq K$ and $J^\circ \subseteq J$ be the sets indexing them. Due to (a) and (b) together with (ii), we would get
       \begin{align*}
           g(x) = \left\langle\sum_{(i^+,i^{-})\in K^\circ} b_{i^{+}}v_{i^+}  + \sum_{j \in J^\circ} b_jv_j,x\right\rangle = 0 .
       \end{align*}
       But this would be a contradiction to $(v,b) \notin \Lambda$. Thus $g_{(i^+,i^-)}=0$ for all $(i^+,i^-) \in K$ and $g_j=0$ for every $j \in J$.  In the first case, $A_{(i^+,i^-)}=\{k_1,k_2\}$ and we must have $c_{k_1}=b_{i^+}$, $v_{k_1}=v_{i^+}$ and $c_{k_2}=b_{i^{-}}$, $v_{k_2}=v_{i^{-1}}$. 
       In the second case $A_{j}=\{k\}$ and we must have $c_k=b_j$ and $w_k=v_j$.
       
       These assignments define a permutation $\tau \in \Pr(\mathbf{d})=\Sigma_{d_1}$ such that  $\tau \cdot (w,c) = (v,b)$, proving identifiability.
\end{proof}
Combining \Cref{prop:shallow-max-identifiable} and \Cref{prop:strongly-identifiable} we obtain the following result.
\begin{corollary} \label{cor:strongly-identifiable-Zariski}
    The $\Sigma_{d_1}$-equivariant semi-algebraic subset $H_0 \subseteq \cP$ is equal to $\cI_{\varphi}$ and to $s\cI_{\varphi}$.
\end{corollary}
\begin{proof}
    From the inclusion $s\cI_{\varphi} \subseteq \cI_{\varphi}$, \Cref{prop:shallow-max-identifiable} and \Cref{prop:strongly-identifiable} we obtain \[
        \cI_{\varphi} \subseteq H_0 \subseteq s\cI_{\varphi} \subseteq \cI_{\varphi}.
    \]
\end{proof}
The following corollary investigates the openness of $\calI_\varphi$: in particular, it shows that in general identifiability is not an open condition in the space of parameters.
\begin{corollary}
    Let $d_0\geq 3$. If $d_1\leq 5$ the set of identifiable parameters $\calI_\varphi$ is Zariski open. If $d_1\geq 6$ $\calI_\varphi$ is not even Euclidean open.
\end{corollary}
\begin{proof}
First, recall that $\Delta$, $\Omega$, $\Lambda$ and $\Xi$ are Zariski closed set. Since $\Upsilon$ is empty for $d_1=2$ and equal to  $\Delta$ for $d_1=3$, it follows that $H_0$ is Zariski open in these cases. We now show that $\Upsilon$ is Zariski closed for $d_1=4$ and $d_1=5$, which proves the first half of the claim.

\underline{Case $d_1=4$} Notice that the points in $(v,b)\in \Upsilon$ in this case come in exactly two forms: either $|K|=2$ and $|J|=0$ or $|K|=1$ and $|J|=2$. In the first case, we can without loss of generality write $K'=\{(1,2)\}$ and $K''=\{(3,4)\}$. Then, in order for $\sum_{i^+\in K'_+}b_{i^+}v_i^+ = b_1v_{1}$ to be in $\mathrm{span}({v_i \, | , i \in K_+''})=\mathrm{span}(v_3)$, we must have $v_1=v_3$ or $v_1=v_4$, which are Zariski-closed conditions.  In the second case, we can without loss of generality assume $K'=\{(1,2)\}$, $J'=\{3\}$ and $J''= \{4\}$, so that the span condition becomes
\begin{align*}
     \sum_{i^+\in K_+'}b_{i^+}v_{i^+}  + \sum_{j \in J'} b_jv_j  = b_1 v_1 + b_3 v_3 \in \mathrm{span}(v_4).
\end{align*}
This is however the case if and only if all minors of the matrix $(b_1v_1+ b_3v_3,v_4)$ are zero, which is a Zariski closed condition.

\underline{Case $d_1=5$} There are again either one or two antipodal pairs. In the first case, we can again without loss of generality assume $K'=\{(1,2)\}$, $J'=\{3,4\}$ and $J''= \{5\}$, so that the span condition becomes 
\begin{align*}
     \sum_{i^+\in K_+'}b_{i^+}v_{i^+}  + \sum_{j \in J'} b_jv_j  = b_1 v_1 + b_3 v_3+b_4v_4 \in \mathrm{span}(v_5),
\end{align*}
which by the same argument as above is a Zariski closed condition. For the case $|K|=2$ we need to distinguish two subcases: $|K'|=1$ or $|K'|=2$. If $K'$ consists of two pairs $(1,2)$ and $(3,4)$, we end up with the condition $b_1v_1 + b_3v_3 \in \mathrm{span}(v_5)$, which again is Zariski closed. If it only consists of one pair $(1,2)$, the condition instead becomes
\begin{align*}
    b_1v_1 \in \mathrm{span}(v_3,v_5).
\end{align*}
But this happens if and only if $v_1,v_3$ and $v_5$ lie in a two-dimensional subspace, which is equivalent to all minors of the matrix $(v_1,v_3,v_5)$ vanishing, i.e.\ again a Zariski closed condition.

\underline{Case $d_1\geq 6$} Now we show that $H_0=\calI_\varphi$ is not euclidean closed for $d_1=6$ and $d_0=3$ (and therefore also for $d_1\geq 6$ and $d_0\ge 4$). We proceed by constructing a one-parameter family of parameters in $\Upsilon$ that converges to a point in $\calI_\varphi$.

Take the three linearly independent standard unit vectors $e_1,e_2,e_3$ and define for $t>0$
    \begin{alignat*}{3}
    v_1 &:= \frac{e_1+e_3}{\sqrt{2}} \quad & v_2 &:= -\frac{e_1+e_3}{\sqrt{2}} \quad & v_3 &:= \frac{e_2-e_3}{\sqrt{2}} \\[1ex]
    v_4 &:= - \frac{e_2-e_3}{\sqrt{2}} \quad & v_5 &:= e_1 \quad & v_6(t) &:= \frac{-e_1 +te_2}{\sqrt{1+t^2}}
\end{alignat*}
By setting $K'=\{(1,2),(3,4)\}$, $K''=\emptyset$ and  $J''=\{5,6\}$, we then get
    \begin{align*}
        \sum_{(i^+,i^{-})\in K'}  b_{i ^+}v_{i^+} = \frac{e_1+e_2}{\sqrt{2}} \in \mathrm{span}(e_1,e_2) =\mathrm{span}(\{v_i \, | \, i\cup J''\})= \mathrm{span}(v_5,v_6(t)).
    \end{align*}
    Thus, $(v_t,b)\in \Upsilon$ for all $t >0$. However, as $t\to0^+$, it converges to the parameter $(\bar{v}, \bar{b})$:
    \begin{align*}
          v_1  := \frac{e_1+e_3}{\sqrt{2}} \quad v_2  :=-\frac{e_1+e_3}{\sqrt{2}}  \quad v_3 := \frac{e_2-e_3}{\sqrt{2}}  \quad v_4= - \frac{e_2-e_3}{\sqrt{2}}  \quad  v_5:= e_1 \quad v_6(0) := -e_1.
    \end{align*}
    The limit $(\bar{v}, \bar{b})$ is not in $\Upsilon$ anymore, since $\mathrm{span}(v_5,v_6(0))= \mathrm{span}(e_1)\not\ni \frac{e_1+e_2}{\sqrt{2}}$. It is simple to check that $(\bar{v}, \bar{b})$ is not contained in any other components of $H_0$, so it belongs to $\calI_{\varphi}$.
\end{proof}

We conclude by constructing the counterexample anticipated in \Cref{sec:semi-algebraicity}.

\begin{example}\label{ex:not-schei}
    Consider a shallow feed-forward ReLU neural network with architecture $(2,2,1)$. We want to show that the morphism $\varphi: (S^1 \times \R)^2 \rightarrow C^0(\R^2)$ does not verify the equivalent conditions in \Cref{theo:schie}. By contradiction, suppose that it exists a closed subset $K \subset  (S^1 \times \R)^2$ as in \Cref{theo:schie}. Consider now any convergent sequence of unit vectors $v_n \rightarrow v$ such that $n(v_n -v) \rightarrow 0$ and $v_n \neq \pm v$ for any $n$. Moreover, consider the sequence of piecewise-linear functions $g_n(x):=n(f_{v_n}-f_v)$, which is clearly contained in $\im \varphi$. Then $g_n  \rightarrow 0$, since 
    \[
        \abs{n(f_{v_n}(x)-f_v(x))}\leq\abs{\sprod{n(v_n-v),x}}.
    \]
    The inequality above is clear on the sets where $f_{v_n}(x)=\sprod{v_n,x}$ and $f_v(x) = \sprod{v,x}$ and where they both are zero. On the sets on which exactly one of the functions is zero, the signs of $\sprod{v,x}$ and $\sprod{v_n,x}$ are different, hence re-introducing the term killed by ReLU will only make the expression bigger.
    As $v_n \neq \pm v$, we know that $g_n \in \varphi(H_0)$, therefore by \Cref{prop:strongly-identifiable} the fiber $\varphi^{-1}(g_n)$ consists of two elements, namely $((v_n,n),(v,-n))$ and $((v,-n),(v_n,n))$ for all $n \in \N$. Since the restriction of $\varphi$ to $K$ is surjective and proper on $\im \varphi$, we know that we should be able to construct a sequence which converges in $(S^1 \times \R)^2$ and maps to $g_n$, up to choosing for every $n$ one of the two elements of fiber $\varphi^{-1}(g_n)$. This is impossible, thus such a closed subset $K$ cannot exists.
\end{example}

\noindent 
\textbf{AI usage statement} We have made use of ChatGPT 5.6 Sol when proving Theorem \ref{th:max-open-in-quotient} and constructing the counterexample to Theorem \ref{th:max-open} for weakly identifiable subsets. All statements and arguments have been verified by the authors, and they take full responsibility for all claims of the paper. In particular,  substantial effort has gone into formulating the constructions in the clearest (to a human) possible way.

{\small 
\bibliography{literature}
\bibliographystyle{alpha} 
}

\newpage
\end{document}